\documentclass{article}
\usepackage[utf8]{inputenc}

\usepackage{tikz}
\usetikzlibrary{arrows.meta, positioning, calc}
\usepackage[normalem]{ulem}
\usepackage{multirow}
\usepackage{natbib}
\usepackage{graphicx}
\usepackage{float}
\usepackage{epstopdf}
\usepackage{amsmath}
\usepackage{amssymb}
\usepackage{amsfonts}
\usepackage{amsthm}
\usepackage{xcolor}
\usepackage{subeqnarray} 
\usepackage[a4paper, total={6in, 9in}]{geometry}
\usepackage[show]{ed}
 \usepackage{hyperref}
 \usepackage{multirow}
 \usepackage{siunitx}
\DeclareSIUnit\mt{\milli\tesla} 
\usepackage{tasks}
\usepackage{xcolor}
\usepackage{colortbl}

\usepackage[ruled,vlined]{algorithm2e} 

\usepackage{subcaption}

\usepackage{booktabs}

\newcommand{\bU}{{\mathbf{U}}}

\newcommand{\bx}{{\mathbf{x}}}

\newcommand{\bu}{{\mathbf{u}}}
\newcommand{\bv}{{\mathbf{v}}}

\newcommand{\br}{{\mathbf{r}}}

\newcommand{\bomega}{{\boldsymbol{\omega}}}

\newcommand{\tr}{{\operatorname{tr}}}
\newcommand{\Reynolds}{{\operatorname{Re}}}

\newcommand\tu{{\tilde{\mathbf{u}}}}

\newcommand\tU{{\tilde{\mathbf{U}}}}

\newcommand\tp{{\tilde{{p}}}}

\newcommand\tomega{{\tilde{\boldsymbol{\omega}}}}

\newtheorem{theorem}{Theorem}[section]
\newtheorem{problem}{Problem}
\newtheorem{lemma}[theorem]{Lemma}

\newtheorem{remark}[theorem]{Remark}

\newtheorem{assumption}[theorem]{Assumption}
\newtheorem{corollary}[theorem]{Corollary}

\newcommand{\be}{\begin{eqnarray}}
\newcommand{\ee}{\end{eqnarray}}
\newcommand{\ben}{\begin{eqnarray*}}
\newcommand{\een}{\end{eqnarray*}}

\title{High-order DLM-ALE discretizations with robust operator preconditioning for  fluid-rigid-body interaction} 
\author{}

\usepackage{authblk}

\author[1,2]{Qi Xin }
\author[1,2]{Shihua Gong \thanks{Corresponding author: gongshihua@cuhk.edu.cn}}
\author[3]{Lingyue Shen } 
\author[3]{Pinjing Wen } 
\author[3]{Yumiao Zhang } 
\author[3]{Yan Chen}
\author[3]{Jiarui Han }
\author[4]{Jinchao Xu}

\affil[1]{School of Science and Engineering, The Chinese University of Hong Kong, Shenzhen, Guangdong 518172, China}
\affil[2]{Shenzhen International Center for Industrial and Applied Mathematics, Shenzhen Research Institute of Big Data, Guangdong 518172, China}
\affil[3]{Shenzhen Raymind Biotechnology Co., Ltd., Shenzhen, Guangdong 518129, China}
\affil[4]{Applied Mathematics and Computational Sciences, CEMSE Division, King Abdullah University of Science and Technology, Thuwal 23955, Saudi Arabia}

\begin{document}
\maketitle
\begin{abstract}
Motivated by the design of deterministic lateral displacement (DLD) microfluidic devices, we develop a high-order numerical framework for fluid-rigid-body interaction on fitted moving meshes. Rigid-body motion is enforced by a distributed Lagrange multiplier (DLM) formulation, while the moving fluid domain is treated by an arbitrary Lagrangian-Eulerian (ALE) mapping. In space, we use isoparametric Taylor-Hood elements to achieve high-order accuracy and to represent curved boundaries and the fluid-particle interface. In time, we employ a high-order partitioned Runge-Kutta strategy in which the mesh motion is advanced explicitly and the coupled physical fields are advanced implicitly, yielding high-order accuracy for the particle trajectory. The fully coupled system is linearized into a generalized Stokes problem subject to distributed constraints of incompressibility and rigid-body motion. We establish well-posedness of this generalized Stokes formulation at both the continuous and discrete levels, providing the stability foundation for operator preconditioning that is robust with respect to key physical and discretization parameters. Numerical experiments on representative benchmarks, including a DLD case, demonstrate high-order convergence for the fluid solution and rigid-body dynamics, as well as robust iterative convergence of the proposed preconditioners.
\end{abstract}

\section{Introduction}
Deterministic Lateral Displacement (DLD) chips have become   a versatile microfluidic platform in biomedical research and diagnostics \cite{huang2004continuous}. In particular, they have been used for the  isolation of circulating tumor cells and separation of blood components \cite{xiang2019precise,liu2018integrated,tang2022topology,gioe2022deterministic,hochstetter2020deterministic}. The underlying principle of DLD relies on a staggered array of micro-pillars placed in a  microchannel, which induces a size-dependent lateral displacement of suspended particles \cite{hochstetter2020deterministic}. As particles traverse  the microchannel, their trajectories bifurcate depending on whether their size is below or above a  critical  diameter, enabling continuous, label-free, and high-throughput separation without external fields.

The critical diameter is a central design quantity in DLD (also called critical particle size), i.e., the particle size at which the trajectory transitions from a zigzag mode to a displaced mode. This threshold is not a material constant; rather, it emerges from the interaction between pillar geometry (shape, spacing, diameter, and stagger angle) and the hydrodynamics of particle-array interaction, and is therefore tunable but difficult to predict a priori. Empirical guidelines and simplified models, often based on ``bump-and-go'' or ``displacement-mode'' assumptions, provide useful intuition but do not fully account for the hydrodynamic coupling between particles and the micro-pillar array  \cite{rezaei2021numerical, pariset2017anticipating}. This motivates high-fidelity numerical simulation as a complementary tool for DLD design.

To resolve particle-array hydrodynamic coupling, we model the particle as a rigid body and represent the fluid-solid interface sharply. This choice is consistent with typical DLD operating conditions, where size-dependent hydrodynamic interactions and geometric effects dominate the separation mechanism and particle deformation is often of secondary importance. The rigid-particle model therefore offers a physically sound and computationally efficient approximation for predicting lateral displacement in DLD devices \cite{rezaei2021numerical,tang2022topology,murmura2019space}.

Accurate prediction of particle trajectories in DLD requires reliable resolution of the fluid-solid interface dynamics, including the kinematic coupling and the associated hydrodynamic forces and torques. A broad class of fluid-structure interaction (FSI) methods has been developed for such problems, ranging from sharp-interface approaches to diffuse-interface formulations. Representative examples include arbitrary Lagrangian-Eulerian (ALE) methods \cite{hirt1974arbitrary,hughes1981lagrangian,xu2015well,li2021energy}, which employ moving meshes to preserve geometric conformity; immersed boundary methods (IBM) \cite{peskin2002immersed,boffi2015finite,boffi2018distributed}, which embed the solid boundary into a fixed fluid grid via distributed forcing; distributed Lagrange multiplier / fictitious domain (DLM/FD) methods \cite{boffi2022existence,boffi2021existence,glowinski1999distributed,glowinski2001fictitious}, which enforce the coupling weakly in a variational framework; and enriched finite element methods \cite{hollbacher2019rotational,hollbacher2020gradient,hollbacher2021sharp}, which incorporate rigid-body motions through enriched trial and test spaces.

In this work, we build on the DLM/FD framework and its variational structure, which naturally leads to a monolithic formulation with saddle-point features. Although the full Navier-Stokes operator yields a nonsymmetric system, the coupled variational formulation provides a convenient setting for analysis and for the construction of operator preconditioners suitable for large-scale computations \cite{boffi2021existence,boffi2022existence,wang2017fictitious,wang2021augmented}.

Our discretization departs from the classical two-grid DLM/FD methodology. In the classical approach, the multiplier is defined on a fictitious domain and one typically interpolates between non-matching fluid and solid meshes. Here we propose a fitted-mesh DLM discretization, in which a single conforming mesh defines the computational domain and the Lagrange multiplier is defined on the solid region and discretized directly on the same mesh. This avoids the interpolation-induced smearing inherent to two-grid coupling. In DLD applications, where the critical diameter is sensitive to trajectory perturbations, such smearing can introduce non-negligible bias in trajectory and threshold prediction. The fitted-mesh DLM formulation therefore combines the variational consistency of DLM with a sharp interface representation that is essential for accurate particle-pillar  hydrodynamic  interactions.

To address these modeling and computational requirements in a unified manner---namely, sharp interface resolution, high-order accuracy, and scalable solvers---we develop a fitted-mesh DLM-ALE framework equipped with operator-based preconditioning. The main contributions of this work are:
\begin{enumerate}
\item  a fitted-mesh DLM discretization for fluid-rigid-body interaction, in which the Lagrange multipliers are discretized on the solid region using the same conforming mesh, eliminating interpolation between non-matching meshes and preserving a sharp fluid-solid interface;
\item  an ALE-based moving-mesh formulation coupled with isoparametric Taylor-Hood elements, enabling high-order geometric representation of curved boundaries and the fluid-particle interface;
\item  a high-order partitioned IMEX Runge-Kutta time integrator that advances mesh motion explicitly while treating the DLM subsystem (momentum balance with incompressibility and rigid-body constraints) implicitly, delivering high-order accurate particle trajectories at reduced cost compared with fully implicit schemes;
\item  a continuous and discrete well-posedness theory for the resulting linearized generalized Stokes system, which provides the stability foundation for robust operator preconditioning;
\item  numerical evidence on benchmark problems and a DLD configuration demonstrating high-order convergence for the fluid solution and rigid-body dynamics, together with robust GMRES convergence of the proposed preconditioners.
\end{enumerate}

The fitted-mesh strategy requires the mesh to conform to the evolving particle boundary. We achieve this by coupling the DLM formulation with an ALE description, in which the mesh motion is computed and updated alongside the physical fields. The resulting formulation is nonlinear and strongly coupled. Fully implicit monolithic strategies offer strong stability but can be computationally demanding \cite{barker2010two,wu2014fully,kong2019simulation,li2021energy}. To reduce cost while preserving robustness, we adopt a partitioned implicit-explicit (IMEX) Runge-Kutta strategy, treating the mesh motion as a non-stiff component and the coupled physical fields as the stiff component. Concretely, the mesh update at each stage is performed explicitly from known stage values, whereas the DLM subsystem—Navier-Stokes momentum balance together with incompressibility and rigid-body constraints—is advanced implicitly. Building on this IMEX partitioning, we design a high-order temporal discretization that enables accurate long-time tracking of particle trajectories at reduced cost.

Efficient solution of the coupled saddle-point linear systems is pivotal for large-scale simulations. In the low-Reynolds-number regime relevant to DLD devices, the convective term can be neglected for the purpose of preconditioner construction. We therefore analyze the well-posedness of the corresponding linearized generalized Stokes system with distributed incompressibility and rigid-body constraints. Our analysis extends the classical theory for Stokes-type saddle-point problems \cite{bramble1997iterative} to the present DLM setting, following the abstract framework developed in \cite{boffi2021existence,boffi2022existence}. The resulting stability estimates motivate a robust block preconditioner based on operator preconditioning ideas for saddle-point systems \cite{loghin2004analysis,Wathen_2015,Benzi_Golub_Liesen_2005,mardal2011preconditioning}.

The remainder of the paper is organized as follows. Section 2 introduces the governing equations and coupling conditions for the fluid-rigid-body interaction model. Section 3 presents the DLM-based variational formulation. Section 4 describes the spatial and temporal discretizations, including isoparametric finite elements, the ALE moving-mesh strategy, and the high-order partitioned Runge-Kutta scheme. Section 5 establishes well-posedness for the linearized problem and develops operator-based preconditioners for the resulting algebraic systems. Section 6 reports numerical results on benchmark problems and a DLD configuration. Section 7 concludes the paper and discusses future directions.

\section{Mathematical model}

This section details the mathematical model for the interaction between a low-Reynolds-number incompressible fluid and a rigid solid. Let $\Omega_f^t \subset \mathbb{R}^d$ and $\Omega_s^t \subset \mathbb{R}^d$ (with $d = 2$ or $3$) denote the time-dependent fluid and solid domains, respectively. These domains evolve within a fixed computational domain $\Omega$ such that $\Omega = \Omega_f^t \cup \Gamma_I^t \cup \Omega_s^t$ for all $t$, where $\Gamma_I^t = \partial \Omega_f^t \cap \partial \Omega_s^t$ is the moving fluid-solid interface. The rigid body is assumed to be fully immersed and non-contact with the external boundary, i.e., $\partial\Omega_s^t \cap \partial \Omega = \emptyset$.

The fluid motion is governed by the incompressible Navier--Stokes equations:
\begin{equation}\label{eq:fluid}
    \begin{aligned}
        \rho_f\frac{D\bu_f}{Dt} - \nabla\cdot\boldsymbol{\sigma}(\bu_f, p) &= \rho_f\mathbf{g}, \quad &\text{in } \Omega_f^t,\\
        \nabla\cdot\bu_f &= 0, \quad &\text{in } \Omega_f^t,
    \end{aligned}
\end{equation}
where $\bu_f$ is the fluid velocity, $p$ is the pressure, $\rho_f$ is the fluid density, and $\mathbf{g}$ is the gravitational acceleration. The operator $D(\cdot)/Dt$ denotes the material derivative, which in the current Eulerian description is given by $\frac{\partial (\cdot)}{\partial t} + (\bu_f \cdot \nabla)(\cdot)$. This form is chosen to facilitate the transition to the ALE framework in subsequent sections. The Cauchy stress tensor for a Newtonian fluid is given by
\begin{equation}\label{eq:stress}
    \boldsymbol{\sigma}(\bu_f, p) = 2\mu_f\varepsilon(\bu_f) - p\mathbf{I},
\end{equation}
where $\mu_f$ is the dynamic viscosity, $\varepsilon(\bu_f) = \frac{1}{2}(\nabla\bu_f + \nabla\bu_f^T)$ is the rate of strain tensor, and $\mathbf{I}$ is the identity tensor. Homogeneous Dirichlet boundary conditions ($\bu_f = \mathbf{0}$) are imposed on the outer boundary $\partial\Omega$.

The motion of the rigid body is governed by the Newton--Euler equations:
\begin{equation}\label{eq:solid}
    \begin{aligned}
        m_s \frac{d\mathbf{U}}{dt} &= \mathbf{F}_{fsi} + m_s \mathbf{g}, \\
        \mathbf{I}_s \frac{d\boldsymbol{\omega}}{dt} + \boldsymbol{\omega} \times (\mathbf{I}_s \boldsymbol{\omega}) &= \mathbf{T}, \\
        \frac{d\mathbf{x}_c}{dt} &= \mathbf{U},
    \end{aligned}
\end{equation}
where $m_s$, $\mathbf{U}(t)$, $\mathbf{F}_{fsi}$, $\mathbf{g}$, $\mathbf{I}_s$, $\boldsymbol{\omega}(t)$, $\mathbf{T}$, and $\mathbf{x}_c(t)$ denote the mass, translational velocity, total hydrodynamic force, gravitational acceleration, moment of inertia tensor, angular velocity, torque, and the position of the center of mass, respectively. 
For a homogeneous rigid body with density $\rho_s$, the mass and moment of inertia are given by
\begin{equation}
    m_s = \int_{\Omega_s^t} \rho_s \, d\mathbf{x}, \quad
    \mathbf{I}_s = \int_{\Omega_s^t} \rho_s \left( |\mathbf{r}|^2 \mathbf{I} - \mathbf{r} \otimes \mathbf{r} \right) d\mathbf{x},
\end{equation}
where $\mathbf{r} = \mathbf{x} - \mathbf{x}_c$ is the position vector relative to the center of mass. 
In this work, we assume that the solid has the same density as the surrounding fluid, i.e., $\rho_s = \rho_f$, corresponding to a neutrally buoyant particle.

The interaction between the fluid and the solid is a two-way coupling: the motion of the solid dictates the flow boundary condition, while the fluid stress, in return, determines the resultant force and torque acting on the solid. This mutual interaction is mathematically encapsulated in the following interface conditions.
The {kinematic condition} requires that the fluid velocity equals the rigid-body surface velocity:
\begin{equation}\label{eq:kinematic_coupling}
    \mathbf{u}_f = \mathbf{U} + \boldsymbol{\omega} \times \mathbf{r}, \quad \mathbf{x} \in \Gamma_I^t.
\end{equation}
The {dynamic condition} enforces the balance of forces and torques between the fluid and the solid:
\begin{equation}\label{eq:dynamic_coupling}
    \mathbf{F}_{fsi} = -\int_{\Gamma_I^t} \boldsymbol{\sigma}(\mathbf{u}_f, p)\, \mathbf{n}\, ds, \qquad
    \mathbf{T} = -\int_{\Gamma_I^t} \mathbf{r} \times \big(\boldsymbol{\sigma}(\mathbf{u}_f, p)\, \mathbf{n}\big)\, ds,
\end{equation}
where $\boldsymbol{\sigma}(\mathbf{u}_f,p)$ denotes the fluid stress tensor and $\mathbf{n}$ is the unit normal vector pointing from the fluid into the solid.

Since the solid undergoes rigid-body motion, the velocity of any material point $\mathbf{x} \in \Omega_s^t$ is given by
\begin{equation}\label{eq:solid_velocity_field}
    \mathbf{u}_s(\mathbf{x},t) = \mathbf{U}(t) + \boldsymbol{\omega}(t) \times (\mathbf{x} - \mathbf{x}_c(t)),
\end{equation}
and the center of mass evolves according to
\begin{equation}\label{eq:cm_evolution}
    \frac{d\mathbf{x}_c}{dt} = \mathbf{U}.
\end{equation}
Consequently, the fluid domain $\Omega_f^t = \Omega \setminus \overline{\Omega_s^t}$ and the fluid-solid interface $\Gamma_I^t = \partial \Omega_f^t \cap \partial \Omega_s^t$ evolve over time according to the rigid-body motion of the solid.

\begin{remark}[Non-dimensionalization]
To simplify the analysis and generalize the results, we non-dimensionalize the governing equations. Let $L$ and $U$ be the characteristic length and velocity of the system, respectively. The dimensionless variables (denoted by asterisks) are defined as:
\begin{equation}
    \begin{aligned}
        \bu_f^* = \frac{\bu_f}{U},\quad
        p^* = \frac{p}{\mu_f U/L},\quad
        \bx^* = \frac{\bx}{L},\quad
        t^* = \frac{tU}{L},\quad
        \boldsymbol{\sigma}^* = \frac{\boldsymbol{\sigma}}{\mu_f U/L},\\
        m_s^* = \frac{m_s}{\rho_f L^d},\quad
        \mathbf{g}^* = \frac{\mathbf{g}L}{U^2},\quad
        \mathbf{F}_{fsi}^* = \frac{\mathbf{F}_{fsi}}{\rho_f U^2 L^{d-1}},\quad
        \mathbf{T}^* = \frac{\mathbf{T}}{\rho_f U^2 L^d},\\
        \mathbf{U}^* = \frac{\mathbf{U}}{U},\quad
        \mathbf{I}_s^* = \frac{\mathbf{I}_s}{\rho_f L^{d+2}},\quad
        \boldsymbol{\omega}^* = \frac{\boldsymbol{\omega}L}{U},\quad
        \mathbf{r}^* = \frac{\mathbf{r}}{L}.
    \end{aligned}
\end{equation}
The resulting dimensionless FSI problem is:
\begin{subequations}
    \begin{align}
        \Reynolds\frac{D\bu_f^*}{Dt^*}  &= \nabla\cdot\boldsymbol{\sigma}^* + \mathbf{g}^*, \quad &\text{in } \Omega_f^t, \label{subeq:ns}\\
        \nabla\cdot\bu_f^* &= 0, \quad &\text{in } \Omega_f^t,\label{subeq:incompressible}\\
        \boldsymbol{\sigma}^* &= 2\varepsilon(\bu_f^*) - p^*\mathbf{I}, \quad &\text{in } \Omega_f^t,\label{subeq:stress}\\
        m_s^*\frac{d\mathbf{U}^*}{dt^*} &= \mathbf{F}_{fsi}^* + m_s^* \mathbf{g}^*, \label{subeq:trans}\\
        \mathbf{I}_s^*\frac{d\boldsymbol{\omega}^*}{dt^*} + \boldsymbol{\omega}^*\times\mathbf{I}_s^*\boldsymbol{\omega}^* &= \mathbf{T}^*, \label{subeq:rot}\\
        \bu_f^* &= \mathbf{U}^* + \boldsymbol{\omega}^*\times \mathbf{r}^*, \quad &\text{on } \Gamma_I^t,\label{subeq:noslip}\\
        \mathbf{F}_{fsi}^* &= -\int_{\Gamma_I^t} \boldsymbol{\sigma}^*\mathbf{n}\ ds,\label{subeq:force} \\
        \mathbf{T}^* &= -\int_{\Gamma_I^t} \mathbf{r}^*\times\boldsymbol{\sigma}^*\mathbf{n}\ ds,\label{subeq:torque}
    \end{align}
\end{subequations}
where the Reynolds number is $\Reynolds = \rho_f U L / \mu_f$. For the remainder of this paper, we work exclusively with this dimensionless system and omit the asterisks for notational simplicity. In typical DLD applications, the Reynolds number $\Reynolds \ll 1$ due to the small length scales and velocities involved. Although we solve the full Navier-Stokes equations, this low Reynolds number regime ensures that the Stokes operator captures the dominant physics, making it an ideal candidate for constructing robust preconditioners.
\end{remark}

\section{Variational Formulation}
\label{sec:variational_formulation}

This section presents the variational formulation that serves as the foundation for our numerical discretization. The primary objective is to derive a {monolithic} system that consistently incorporates the fluid-structure coupling. Our key contribution is a novel formulation that combines the geometric precision of a fitted mesh with the mathematical rigor of the Distributed Lagrange Multiplier (DLM) framework. We first introduce the necessary function spaces. For any domain $w \subset \mathbb{R}^d$, let $H^1(w)$ denote the Sobolev space of functions in $L^2(w)$ with first-order weak derivatives in $L^2(w)$, and let $H^1_0(w)$ be the subspace with zero trace on $\partial w$. Let $L^2_0(w)$ denote the space of $L^2(w)$ functions with zero mean.

The final formulation, presented immediately below, employs a Lagrange multiplier field defined over the solid region and discretized using a single, conforming mesh, thereby avoiding the interpolation errors associated with classical two-grid DLM methods.

\subsection{Monolithic DLM Formulation}

We begin by stating the final variational problem, which is the cornerstone of our computational approach. Let $\boldsymbol{\Lambda}$ be a Hilbert space for the Lagrange multiplier, chosen here as $\boldsymbol{\Lambda} = [H^1(\Omega_s^t)]^d$.

\begin{problem}[Monolithic DLM Formulation] \label{prob:final_dlm_form}
Find the global velocity $\bu \in [H^1_0(\Omega)]^d$, the fluid pressure $p \in L^2_0(\Omega_f^t)$, the rigid-body translational and rotational velocities $\bU, \boldsymbol{\omega} \in \mathbb{R}^d$, and the Lagrange multiplier $\boldsymbol{\lambda} \in \boldsymbol{\Lambda}$ such that
\begin{equation}\label{eq:final_monolithic_system}
    \begin{aligned}
        \Reynolds\, m^t\!\left( \frac{D\bu}{Dt}, \bv \right) + a^t(\bu,\bv) + b^t(p,\bv) + c^t(\boldsymbol{\lambda}, \bv - \mathbf{V} - \boldsymbol{\xi} \times \br) &= m^t(\mathbf{g},\bv)_, \\
        b^t(q, \bu) &= 0, \\
        c^t(\boldsymbol{\mu}, \bu - \bU - \boldsymbol{\omega} \times \br) &= 0,
    \end{aligned}
\end{equation}
for all test functions $\bv \in [H^1_0(\Omega)]^d$, $q \in L^2_0(\Omega_f^t)$, $\mathbf{V}, \boldsymbol{\xi} \in \mathbb{R}^d$, and $\boldsymbol{\mu} \in \boldsymbol{\Lambda}$. Here, $\br = \bx - \bx_c$, the bilinear forms are defined as:
\begin{equation}\label{eq:final_bilinear_forms}
    \begin{aligned}
        m^t(\bu,\bv) &= \int_{\Omega} \bu \cdot \bv  \,d\bx, \\
        a^t(\bu,\bv) &= 2 \int_{\Omega} \varepsilon(\bu) : \varepsilon(\bv)  \,d\bx, \\
        b^t(p,\bv) &= -\int_{\Omega_f^t} p \, \nabla \cdot \bv  \,d\bx, \\
        c^t(\boldsymbol{\lambda}, \boldsymbol{\mu}) &= \alpha \int_{\Omega_s^t} \boldsymbol{\lambda} \cdot \boldsymbol{\mu}  \,d\bx + 2 \int_{\Omega_s^t} \varepsilon(\boldsymbol{\lambda}) : \varepsilon(\boldsymbol{\mu})  \,d\bx,
    \end{aligned}
\end{equation}
and the evolution of the domain $\Omega_f^t$ and $\Omega_s^t$ is governed by the rigid-body motion:
\begin{equation}
    \frac{d\bx_c}{dt} = \bU, \quad \frac{d\bx}{dt} = \bU + \boldsymbol{\omega} \times (\bx - \bx_c), \quad \bx \in \Omega_s^t, \quad \Omega_f^t = \Omega \setminus \overline{\Omega_s^t}.
\end{equation}

The parameter $\alpha > 0$ is related to the temporal discretization and influences the stability of the formulation; a detailed analysis of its selection will be presented in Section 5.
\end{problem}

This formulation possesses several advantageous features crucial for our computational framework. First, it is posed on the entire, fixed domain $\Omega$, which facilitates the use of a single, conforming mesh and circumvents the complexities associated with tracking the evolving fluid domain $\Omega_f^t$. Second, the kinematic constraint $\bu = \bU + \boldsymbol{\omega} \times \br$ in $\Omega_s^t$ is enforced not through the function space, but weakly by the Lagrange multiplier $\boldsymbol{\lambda}$ via the term $c^t(\boldsymbol{\mu}, \bu - \bU - \boldsymbol{\omega} \times \br)$. This approach liberates the discretization from the rigid-body constraint, allowing for the use of standard finite element spaces. Lastly, and most importantly, the specific structure of the bilinear form $c^t(\cdot, \cdot)$ defined in \eqref{eq:final_bilinear_forms} is carefully designed to ensure numerical stability. Its coercivity in a norm strong enough to control the rigid body modes is instrumental for proving the well-posedness of the discrete problem and for facilitating the construction of robust block preconditioners for the resulting saddle-point systems, which is essential for computational efficiency in large-scale simulations.

\begin{remark}\label{rem:density}
    In the case where the solid density $\rho_s$ differs from the fluid density $\rho_f$, the formulation can be readily adapted by modifying the mass bilinear form to account for the density contrast. Specifically, we redefine $m^t(\cdot, \cdot)$ as:
    \begin{equation}
        m^t(\bu,\bv) = \int_{\Omega_f^t} \bu \cdot \bv  \,d\bx + \frac{\rho_s}{\rho_f} \int_{\Omega_s^t} \bu \cdot \bv  \,d\bx.
    \end{equation}
\end{remark}

\subsection{Derivation of the Monolithic DLM Formulation}

The derivation of Problem~\ref{prob:final_dlm_form} proceeds from a classical formulation with strong constraints to the final weak formulation. We first formulate the problem in the Eulerian frame on the time-dependent fluid domain $\Omega_f^t$. To enforce the kinematic coupling, we define the constrained space:
\begin{equation}\label{eq:restricted_space}
H_{0,R}^1(\Omega_f^t) = \left\{ \bv_f \in [H^1(\Omega_f^t)]^d \,\middle|\,
\bv_f = \mathbf{V} + \boldsymbol{\xi} \times \mathbf{r} \text{ on } \Gamma_I^t,\ \bv_f = 0 \text{ on } \partial\Omega,\ \text{for some } \mathbf{V}, \boldsymbol{\xi} \in \mathbb{R}^d
\right\}.
\end{equation}

Assuming $\bu_f \in H_{0,R}^1(\Omega_f^t)$ and noting that the kinematic coupling condition $\bu_f = \mathbf{U} + \boldsymbol{\omega} \times \mathbf{r}$ holds on $\Gamma_I^t$, the standard procedure of multiplying the momentum equation~\eqref{subeq:ns} by a test function $\bv_f \in H_{0,R}^1(\Omega_f^t)$, integrating by parts, and substituting the dynamic coupling conditions~\eqref{subeq:noslip} and Newton-Euler equations~\eqref{subeq:force}~\eqref{subeq:torque} leads to the following formulation: find $(\bu_f, p, \mathbf{U}, \boldsymbol{\omega})$ with $\bu_f \in H_{0,R}^1(\Omega_f^t)$, $p \in L^2_0(\Omega_f^t)$ such that for all $\bv_f \in H_{0,R}^1(\Omega_f^t)$ and $q \in L^2_0(\Omega_f^t)$:
\begin{equation}\label{eq:weak_system_fluid_domain}
\begin{aligned}
\Reynolds \int_{\Omega_f^t} \frac{D\bu_f}{Dt} \cdot \bv_f d\bx
+ 2\mu_f \int_{\Omega_f^t} \varepsilon(\bu_f) : \varepsilon(\bv_f) d\bx
- \int_{\Omega_f^t} p \nabla \cdot \bv_f d\bx \\
+ m_s \left( \frac{d\mathbf{U}}{dt} - \mathbf{g} \right) \cdot \mathbf{V}
+ \left( \mathbf{I}_s \frac{d\boldsymbol{\omega}}{dt} + \boldsymbol{\omega} \times (\mathbf{I}_s \boldsymbol{\omega}) \right) \cdot \boldsymbol{\xi}
&= \int_{\Omega_f^t} \mathbf{g} \cdot \bv_f d\bx, \\
\int_{\Omega_f^t} q \nabla \cdot \bu_f d\bx &= 0,
\end{aligned}
\end{equation}
where $\mathbf{V}$ and $\boldsymbol{\xi}$ are determined by $\bv_f|_{\Gamma_I^t} = \mathbf{V} + \boldsymbol{\xi} \times \mathbf{r}$.

While conceptually straightforward, this formulation is computationally challenging due to the evolving domain $\Omega_f^t$ and the cumbersome constrained function space.

A more unified perspective is obtained by extending the fluid velocity to the entire domain $\Omega$. Define the global velocity field $\bu : \Omega \to \mathbb{R}^d$ where $\bu$ equals the fluid velocity $\bu_f$ in $\Omega_f^t$ and equals the rigid-body velocity $\mathbf{U} + \boldsymbol{\omega} \times \br$ in $\Omega_s^t$. The kinematic coupling ensures $\bu$ is continuous across $\Gamma_I^t$. We now introduce the unified constrained space:
\begin{equation}\label{eq:unified_test_space}
H_{0,R}^1(\Omega) = \left\{ \bv \in [H_0^1(\Omega)]^d \,\middle|\, 
\bv|_{\Omega_s^t} = \mathbf{V} + \boldsymbol{\xi} \times \br,\ 
\mathbf{V}, \boldsymbol{\xi} \in \mathbb{R}^d \right\}.
\end{equation}

A key insight is that for any $\bu$ and $\bv$ in this space, the following identity holds within the solid domain $\Omega_s^t$ due to the properties of rigid-body motions:
\begin{equation}\label{eq:key_identity}
\Reynolds \int_{\Omega_s^t} \frac{D\bu}{Dt} \cdot \bv d\bx + 2 \int_{\Omega_s^t} \varepsilon(\bu) : \varepsilon(\bv) d\bx = m_s \frac{d\bU}{dt} \cdot \mathbf{V} + \left( \mathbf{I}s \frac{d\boldsymbol{\omega}}{dt} + \boldsymbol{\omega} \times \mathbf{I}_s \boldsymbol{\omega} \right) \cdot \boldsymbol{\xi}.
\end{equation}
This identity allows us to absorb the solid inertia terms in \eqref{eq:weak_system_fluid_domain} into domain integrals over $\Omega$, yielding an equivalent but more compact formulation: find $\bu \in H_{0,R}^1(\Omega)$ and $p \in L^2_0(\Omega_f^t)$ such that
\begin{equation}\label{eq:weak_system_unified_domain}
\begin{aligned}
\Reynolds\int_{\Omega} \frac{D\bu}{Dt}\cdot\bv\ d\bx + 2\int_{\Omega} \varepsilon(\bu):\varepsilon(\bv)\ d\bx - \int_{\Omega_f^t} p\ \nabla\cdot\bv\ d\bx &= \int_{\Omega} \mathbf{g}\cdot\bv\ d\bx,\\
\int_{\Omega_f^t} q\ \nabla\cdot\bu\ d\bx &= 0,
\end{aligned}
\end{equation}
for all $\bv \in H_{0,R}^1(\Omega)$ and $q \in L^2_0(\Omega_f^t)$.

This form elegantly expresses the coupled FSI problem within a single variational framework. However, it still relies on a function space that is constrained to rigid motions in $\Omega_s^t$, which complicates the use of standard finite element discretizations on a fixed mesh.

To circumvent this, we employ the Distributed Lagrange Multiplier method. The core idea is to relax the constraint $\bu = \bU + \boldsymbol{\omega} \times \br$ in $\Omega_s^t$ and enforce it weakly by introducing a Lagrange multiplier field $\boldsymbol{\lambda} \in \boldsymbol{\Lambda}$.
We define the saddle-point problem on the unconstrained space $[H^1_0(\Omega)]^d \times \boldsymbol{\Lambda}$, which leads directly to the final formulation stated in Problem~\ref{prob:final_dlm_form}. 
The specific construction of the bilinear form $c(\boldsymbol{\lambda}, \boldsymbol{\mu})$ in \eqref{eq:final_bilinear_forms} is carefully designed to ensure numerical stability. In general, for the Distributed Lagrange Multiplier method to be well-posed, the Lagrange multiplier space $\boldsymbol{\Lambda}$ and the bilinear form $c^t : \boldsymbol{\Lambda} \times [H^1(\Omega_s^t)]^d \to \mathbb{R}$ must satisfy certain conditions:
\begin{equation*}
\begin{aligned}
&c^t \text{ is bounded on } \boldsymbol{\Lambda} \times [H^1(\Omega_s^t)]^d, \
&c^t(\boldsymbol{\mu}, \bv) = 0 \quad \forall, \boldsymbol{\mu} \in \boldsymbol{\Lambda} \quad \text{implies} \quad \bv = 0 \quad \text{in } \Omega_s^t.
\end{aligned}
\end{equation*}
This ensures that the constraint imposed by the Lagrange multiplier is complete: the only function in $[H^1(\Omega_s^t)]^d$ that is $c^t$-orthogonal to all $\boldsymbol{\mu} \in \boldsymbol{\Lambda}$ is the zero function.

Common choices include taking $\boldsymbol{\Lambda} = ([H^1(\Omega_s^t)]^d)'$ with $c^t$ as the duality pairing, or $\boldsymbol{\Lambda} = [H^1(\Omega_s^t)]^d$ with $c^t$ as the $H^1$ inner product. In our formulation, we adopt the latter approach with $\boldsymbol{\Lambda} = [H^1(\Omega_s^t)]^d$ and define $c^t$ as a weighted combination:
\begin{equation*}
c^t(\boldsymbol{\lambda}, \boldsymbol{\mu}) = \alpha \int_{\Omega_s^t} \boldsymbol{\lambda} \cdot \boldsymbol{\mu} d\bx + 2 \int_{\Omega_s^t} \varepsilon(\boldsymbol{\lambda}) : \varepsilon(\boldsymbol{\mu}) d\bx.
\end{equation*}
The parameter $\alpha > 0$ is related to the temporal discretization and influences the stability of the formulation; a detailed analysis of its selection will be presented in Section 5. 

This derivation demonstrates that our proposed monolithic DLM formulation (Problem~\ref{prob:final_dlm_form}) is mathematically equivalent to the classical strong-constraint formulations, yet it is far more friendly to numerical implementation using our fitted-mesh strategy.

\section{Discretisation schemes}
\subsection{Formulation in reference configuration}

Before we start the discussion of the discretisation, we discuss the formulation in an arbitrary Lagrangian-Eulerian coordinate system. We follow the idea in \cite{gawlik2014high} and \cite{gawlik2015unified} to transform the problem into a fixed domain. For any $t\in[0,T]$, let $\mathcal{A}^t:\Omega \to \Omega$ be a diffeomorphism satisfying  $\Gamma_I^t = \mathcal{A}^t(\Gamma_I^0)$. The fluid domain is given by $\Omega_f^t = \mathcal{A}^t(\Omega_f^0)$ and the solid domain is given by $\Omega_s^t = \mathcal{A}^t(\Omega_s^0)$.  
We introduce the following function spaces in the reference configuration:
\begin{equation}
    \tilde{V} = [H^1_0(\Omega)]^d,\quad \tilde{Q} = L^2_0(\Omega_f^0),\quad \tilde{\boldsymbol{\Lambda}} =[H^1(\Omega_s^0)]^d.
\end{equation}
To avoid making the problem too complicated at the moment, we first assume that the mapping $\mathcal{A}^t$ is known. 
We define the following variables in the reference configuration:
\begin{equation}\label{eq:equivalence_coordinate}
    \begin{aligned}
    \tu(\tilde{\bx},t) &= \bu(\mathcal{A}^t(\tilde{\bx}),t),\quad\\
    \tp(\tilde{\bx},t) &= p(\mathcal{A}^t(\tilde{\bx}),t),\quad\\
    \tilde{\boldsymbol{\lambda}}(\tilde{\bx},t) &= \boldsymbol{\lambda}(\mathcal{A}^t(\tilde{\bx}),t).
\end{aligned}
\end{equation}
Although $\bU$ and $\boldsymbol{\omega}$ are independent of the spatial variable, we still use $\tU$ and $\tomega$ to denote them in the reference configuration for notational consistency. 
The bilinear forms in the reference configuration are given by
\begin{equation*}
    \begin{aligned}
        & {m}^t(\tu, \tilde{\bv}) = \int_{\Omega} \tu\cdot\tilde{\bv}|\det(F)| \ d\tilde{\bx},\quad
        \frac{\tilde{D}\tu}{\tilde{D}t} = \frac{\partial\tu}{\partial t} + (\tu - \frac{\partial \mathcal{A}^t}{\partial t})\cdot \nabla_{\bx}\tu,\\
        &{a}^t(\tu,\tilde{\bv}) = \int_{\Omega} \varepsilon(\tu):\varepsilon(\tilde{\bv})|\det(F)| \ d\tilde{\bx},
        \quad\varepsilon(\tu) = \frac{1}{2}(\nabla_{\bx}\tu + (\nabla_{\bx}\tu)^T),\\
        &{b}^t(\tp, \tilde{\bv}) = \int_{\Omega_f^t} \tp\ \tr(\nabla_{\bx}\tilde{\bv})|\det(F)| \ d\tilde{\bx},\quad \nabla_{\bx}\tilde{\bv} = \nabla_{\tilde{\bx}} \tilde{\bv}\ F^{-1},\\
        &F = \nabla_{\tilde{\bx}} \mathcal{A}^t.
    \end{aligned}
\end{equation*}

Through this change of variables, the ``evolution of the domain'' is now mathematically encapsulated in the nonlinear terms of the governing equations involving $\mathcal{A}^t$ and its temporal derivative (the mesh velocity $\partial_t \mathcal{A}^t$). This transformation is pivotal because it represents the evolving geometric domain through an explicit mapping variable. This formulation clarifies the separation between mesh movement and physical evolution, providing a transparent framework for high-order temporal discretization.
For notational simplicity and algorithmic clarity, we recast the coupled system \eqref{eq:final_monolithic_system} into the compact abstract form \eqref{star} by introducing the composite solution vector $\Phi = (\tu, \tp, \tilde{\bU}, \tilde{\boldsymbol{\omega}}, \tilde{\boldsymbol{\lambda}})$, test function vector $\Psi = (\tilde{\bv},\tilde{q},\tilde{\mathbf{V}} , \tilde{\boldsymbol{\xi}}, \tilde{\boldsymbol{\mu}})$, and the nonlinear operator $\mathcal{F}$ encapsulating the weak form of the governing equations:
\begin{equation}\label{star}
    \mathcal{F}(\frac{\tilde{D}\tu}{\tilde{D}t},\Phi, \Psi,\mathcal{A}^t)=0,
\end{equation}
where 
\begin{equation}\label{eq:Full_form_ALE}
    \begin{aligned}
        \mathcal{F}(\frac{\tilde{D}\tu}{\tilde{D}t},\Phi, \Psi,\mathcal{A}^t) =
        &\Reynolds ~{m}^t(\frac{\tilde{D}\tu}{\tilde{D}t}, \tilde{\bv}) + {a}^t(\tu,\tilde{\bv}) + {b}^t(\tp,\tilde{\bv}) + c^t(\tilde{\boldsymbol{\lambda}}, \tilde{\bv} - \tU - \tomega \times \br) - {m}^t(\mathbf{g},\tilde{\bv})\\
        &+ {b}^t(\tilde{q}, \tu) + c^t(\tilde{\boldsymbol{\mu}}, \tu - \tU - \tomega \times \br).
    \end{aligned}
\end{equation}

\subsection{Spatial discretisation}
We use the finite element method to discretise the problem. We denote $\mathcal{T}_h$ as a triangulation of the domain $\Omega$ with mesh size $h$. We denote $\tilde{V}_h$ as the finite element space of degree $2$ for the velocity field, and $\tilde{Q}_h$ as the finite element space of degree $1$ for the pressure field. To enforce the coupling conditions, we choose $\tilde{\boldsymbol{\Lambda}}_h$ to be the finite element space of the same degree as the velocity field. We use the isoparametric finite element to approximate the domain $\Omega$ with curved boundaries and interface $\Gamma_I^t$.
The finite element spaces are defined as
\begin{equation}
    \begin{aligned}
        \tilde{V}_h &= \{ \bv_h\in \tilde{V} ~| ~\bv_h|_K\circ F_K\in [P_2(\hat{K})]^d,\ K\in\mathcal{T}_h\},\\
        \tilde{Q}_h &= \{ q_h\in \tilde{Q} ~| ~q_h|_K\circ F_K\in P_1(\hat{K}),\ K\in\mathcal{T}_h\},\\
        \tilde{\boldsymbol{\Lambda}}_h &= \{ \boldsymbol{\lambda}_h\in \tilde{\boldsymbol{\Lambda}} ~| ~\boldsymbol{\lambda}_h|_K\circ F_K\in [P_2(K)]^d,\ K\in\mathcal{T}_h\},
    \end{aligned}
\end{equation}
where $P_2(\hat{K})$ and $P_1(\hat{K})$ are the polynomial spaces of degree $2$ and $1$ on the reference element $\hat{K}$, respectively. The mapping $F_K:\hat{K}\to K$ is the isoparametric mapping from the reference element to the physical element, i.e., \begin{equation}F_K(\hat{\bx}) = \sum_{i=1}^{N_K} \hat{\phi}_i(\hat{\bx})\mathbf{x}_i,\end{equation} where $\hat{\phi}_i$ are the shape functions (or nodal basis functions) on the reference element and $\mathbf{x}_i$ are the nodal points of the physical element. In this way, the triangluation $\mathcal{T}_h$ is defined by the nodal points $\mathbf{x}_i$ and the shape functions $\hat{\phi}_i$, and the physical element $K$ can approximate the domain $\Omega$ with curved boundaries and interface $\Gamma_I^t$ with the accuracy of $O(h^3)$. The isoparametric element is crucial for the accuracy of the discretisation for the FSI problem, as the interface $\Gamma_I^t$ is not a piecewise linear curve in general. The isoparametric mapping allows us to use the same shape functions for both the physical element and the reference element, which ensures that the interface is approximated with the same accuracy as the fluid velocity \cite{brenner2008mathematical}.

The finite element approximation of the problem is given by
\begin{problem}
Given $\mathcal{A}^t$ and $\partial_t \mathcal{A}^t$, find the generalized solution 
\[
{\Phi}_h =(\tu_h, \tp_h, \tilde{\bU}_h, \tilde{\boldsymbol{\omega}}_h, \tilde{\boldsymbol{\lambda}}_h) \in \tilde{V}_h \times \tilde{Q}_h \times \mathbb{R}^d \times \mathbb{R}^d \times \tilde{\boldsymbol{\Lambda}}_h
\]
such that
\begin{equation}\label{eq:general_ALE}
\mathcal{F}\left( 
\frac{\partial \tu_h}{\partial t} + (\tu_h - \partial_t \mathcal{A}^t) \cdot \nabla_{\bx} \tu_h,\ 
{\Phi}_h,\ 
{\Psi}_h,\ 
\mathcal{A}^t 
\right) = 0.
\end{equation}
for all $\boldsymbol{\Psi}_h \in \tilde{V}_h \times \tilde{Q}_h \times \mathbb{R}^d \times \mathbb{R}^d \times \tilde{\boldsymbol{\Lambda}}_h$, where $\mathcal{F}$ is defined in \eqref{eq:Full_form_ALE}.
\end{problem}
\begin{remark}
  This semi-discretized problem is a dynamical system in time defined on fixed finite element spaces. By parameterizing the domain motion as a time-dependent variable, this approach not only eradicates the difficulty of dealing with time-dependent fluid domains but also clarifies the partitioning logic in temporal schemes, such as the IMEX strategies discussed in Section 4.4.
\end{remark}

\subsection{Construction of the domain mapping}
For a realistic FSI problem, the domain mapping $\mathcal{A}^t$ is unknown and coupled with the velocity field. To maintain geometric consistency as the solid moves, the domain mapping must respect the rigid-body motion of the embedded object. This imposes a boundary condition on the mapping's evolution:
\begin{equation}
    \left\{
    \begin{aligned}
    \frac{d\mathcal{A}^t(\tilde{\mathbf{x}})}{dt} &= \tU + \tomega\times\mathbf{r}, &\text{ on } \partial\Omega_s^0,\\
    \mathcal{A}^0(\tilde{\mathbf{x}}) &= \tilde{\mathbf{x}},&\text{in }\Omega.
    \end{aligned}
    \right.
\end{equation}
However, this condition only prescribes the motion of the boundary. To define a unique mapping throughout the entire domain $\Omega$, this boundary velocity must be extended to the interior. This extension is not unique, and various strategies exist, such as solving a Laplace (or pseudo-elasticity) equation for the mesh nodal displacements or using radial basis functions. The choice of extension algorithm represents a modeling decision that affects mesh quality and numerical stability, but does not alter the underlying physics of the fluid-structure interaction.

A wide range of techniques exist for extending the interface motion to the entire fluid domain \cite{jasak2006automatic,rangarajan2019algorithm,ramsharan2011analysis,barral2014two,sackinger1996newton,le2001fluid,cerroni2014new}. In this work, we employ a {harmonic extension} method, which solves an elliptic PDE on the reference domain to smoothly propagate boundary displacements into the interior while minimizing mesh distortion.

Let $\mathcal{A}^t(\tilde{\mathbf{x}}) = \tilde{\mathbf{x}} + \mathbf{X}_h^t(\tilde{\mathbf{x}})$, where $\mathbf{X}_h^t$ denotes the mesh displacement field at time $t$. Given a target displacement increment on the solid interface (or rigid body), we define a general mesh update problem as follows.

\begin{problem}\label{prob:mesh_update}
Given $(\mathbf{X}_h^{\rm prev}, \Delta t, \alpha, \tU, \tomega, \mathbf{x}_c)$ representing the previous mesh displacement, time step size, time stepping parameter, rigid-body translational and rotational velocities, and center of mass position, respectively, find the updated displacement $\mathbf{X}_h \in Y$ such that
\begin{equation}\label{eq:mesh_general}
    \begin{aligned}
    (\nabla \mathbf{X}_h, \nabla \mathbf{Y}) &= 0 
        &\quad &\text{in } \Omega_f^0, \\
    \mathbf{X}_h &= \mathcal{G}(\mathbf{X}_h^{\rm prev}; \tU, \tomega, \alpha\Delta t)
        &\quad &\text{on } \Omega_s^0, \\
    \mathbf{X}_h &= 0 
        &\quad &\text{on } \partial\Omega,
    \end{aligned}
\end{equation}
for all test functions $\mathbf{Y} \in Y_0$, where the rigid-body transformation $\mathcal{G}$ is defined as:
\[
\mathcal{G}(\mathbf{y}; \tU, \tomega, \alpha\Delta t) = \mathbf{y} + \alpha\Delta t\tU + (\mathbf{Q}(\alpha\Delta t\tomega)-\mathbf{I}) \mathbf{r},
\]
with $\mathbf{Q}(\alpha\Delta t\tomega)$ being the rotation matrix corresponding to the angular displacement $\alpha\Delta t\tomega$ defined below:
\begin{equation}\label{eq:rotation_matrix}
    \mathbf{Q}(\boldsymbol{\theta}) = \begin{cases}
    \mathbf{I} + \frac{\sin(|\boldsymbol{\theta}|)}{|\boldsymbol{\theta}|}\hat{\boldsymbol{\theta}} + \frac{(1 - \cos(|\boldsymbol{\theta}|))}{|\boldsymbol{\theta}|^2}\hat{\boldsymbol{\theta}}^2, & d=3,\\
    \begin{pmatrix}
    \cos(\theta) & -\sin(\theta) \\
    \sin(\theta) & \cos(\theta)
    \end{pmatrix}, & d=2,
    \end{cases}
\end{equation}
where 
\begin{itemize}
\item for $\boldsymbol{\theta} = (\theta_1, \theta_2, \theta_3)^T \in \mathbb{R}^3$ in three-dimensional case, $\hat{\boldsymbol{\theta}}$ is the skew-symmetric matrix defined as
\[
\hat{\boldsymbol{\theta}} = \begin{pmatrix}
0 & -\theta_3 & \theta_2 \\
\theta_3 & 0 & -\theta_1 \\
-\theta_2 & \theta_1 & 0
\end{pmatrix}.
\]
\item for two-dimensional case, $\boldsymbol{\theta} = \theta \mathbf{e}_3$ with $\mathbf{e}_3$ being the unit vector perpendicular to the plane of motion.
\end{itemize}
The finite element spaces are defined as:
\begin{equation}
    \begin{aligned}
    Y &= \left\{ \mathbf{X}_h \in [H^1_0(\Omega)]^d \;\middle|\; \mathbf{X}_h|_K \in [P_1(K)]^d\ \forall K \in \mathcal{T}_h \right\}, \\
    Y_0 &= \left\{ \mathbf{Y}_h \in [H^1_0(\Omega)]^d \;\middle|\; \mathbf{Y}_h|_K \in [P_1(K)]^d\ \forall K \in \mathcal{T}_h,\ \mathbf{Y}_h|_{\Omega_s^0} = 0 \right\}.
    \end{aligned}
\end{equation}
\end{problem}
\begin{remark}
Here we use linear finite elements for computational efficiency and conservation of the triangulation away from the boundary. However, the curved edges of the isoparametric elements on the interface $\Gamma_I^t$ needs to be updated according to the new position of the solid to maintain geometric accuracy.
\end{remark}

\subsection{Temporal discretisation}
The fully coupled FSI-ALE system, after spatial discretization, can be expressed as a system of ordinary differential equations (ODEs) in time. To facilitate the application of partitioned Runge-Kutta methods, we rewrite the semi-discrete system in a more abstract form, and separate the non-stiff mesh update from the potentially stiff FSI dynamics. For this purpose, we introduce the following notation.
Let $\mathbf{Y}(t)$ denote the vector of all physical variables at time $t$, including fluid velocity, pressure, rigid-body velocities, and Lagrange multipliers, and let $\mathbf{X}(t)$ represent the mesh configuration at time $t$. The semi-discrete FSI-ALE system can be expressed as:
\begin{equation}
    \left\{
    \begin{aligned}
        \frac{d\mathbf{X}}{dt} &= \mathbf{F}(\mathbf{X}, \mathbf{Y}), \\
        \frac{d\mathbf{Y}}{dt} &= \mathbf{G}(\mathbf{X}, \mathbf{Y}),
    \end{aligned}
    \right.
\end{equation}
where $\mathbf{G}$ encapsulates the fluid-structure interaction dynamics, and $\mathbf{F}$ describes the mesh motion driven by the rigid-body dynamics. Then we can apply a partitioned Runge-Kutta method to this system, treating the mesh update and the FSI solve as separate stages. These methods, which belong to the broader class of additive Runge-Kutta methods, employ different Butcher tableaux for different components of the right-hand side function. For a detailed discussion on partitioned and additive Runge-Kutta methods, we refer to \cite{kennedy2003additive,ascher1997implicit}.
The general form of a partitioned Runge-Kutta method for this system can be written as
\begin{equation}
    \left\{
    \begin{aligned}
        \mathbf{X}_{n+1} &= \mathbf{X}_n + \Delta t \sum_{i=1}^{s+1} \tilde{b}_i \mathbf{F}(\mathbf{X}_n^{(i-1)}, \mathbf{Y}_n^{(i-1)}), \\
        \mathbf{Y}_{n+1} &= \mathbf{Y}_n + \Delta t \sum_{i=1}^s {b}_i \mathbf{G}(\mathbf{X}_n^{(i)}, \mathbf{Y}_n^{(i)}),
    \end{aligned}
    \right.
\end{equation}
with $\mathbf{X}_n^{(i)}, \mathbf{Y}_n^{(i)}$ being the stage values defined by
\begin{equation}
    \left\{
    \begin{aligned}
        \mathbf{X}_n^{(0)} &= \mathbf{X}_n, \quad \mathbf{Y}_n^{(0)} = \mathbf{Y}_n,\\
        \mathbf{X}_n^{(i)} &= \mathbf{X}_n + \Delta t \sum_{k=1}^{i+1} \tilde{a}_{i+1,k} \mathbf{F}(\mathbf{X}_n^{(k-1)}, \mathbf{Y}_n^{(k-1)}), \\
        \mathbf{Y}_n^{(i)} &= \mathbf{Y}_n + \Delta t \sum_{k=1}^i a_{i,k} \mathbf{G}(\mathbf{X}_n^{(k)}, \mathbf{Y}_n^{(k)}).
    \end{aligned}
    \right.
\end{equation}
The coefficients $\tilde{a}_{i,j}, \tilde{b}_i, a_{i,j}, b_i$ define the specific Runge-Kutta method used for each subsystem, where $\tilde{a}_{i,j}, \tilde{b}_i$ correspond to the explicit part (mesh motion) and $a_{i,j}, b_i$ correspond to the implicit part (FSI solve). The corresponding Butcher tableaux are given by:
\begin{center}
\begin{minipage}{0.45\textwidth}
\centering
\textbf{Implicit part}
\begin{tabular}{c|ccccc}
    0 & 0 & 0 & 0 & $\cdots$ & $0$\\
  $c_1$ & 0 & $a_{1,1}$ & $0$ & $\cdots$ & $0$\\
  $c_2$ & 0 & $a_{2,1}$ & $a_{2,2}$ & $\cdots$ & $0$\\
  $\vdots$ & $\vdots$ & $\vdots$ & $\vdots$ & $\ddots$ & $\vdots$\\
  $c_s$ & 0 & $a_{s,1}$ & $a_{s,2}$ & $\cdots$ & $a_{s,s}$\\
  \hline
    & 0 & $b_1$ & $b_2$ & $\cdots$ & $b_s$ \\
\end{tabular}
\end{minipage}
\hfill
\begin{minipage}{0.45\textwidth}
\centering
\textbf{Explicit part}
\begin{tabular}{c|ccccc}
  $0$ & $0$ & $0$ & $\cdots$ & $0$\\
  $c_1$ & $\tilde{a}_{2,1}$ & $0$ & $\cdots$ & $0$\\
  $\vdots$ & $\vdots$ & $\vdots$ & $\ddots$ & $\vdots$\\
  $c_s$ & $\tilde{a}_{s+1,1}$ & $\tilde{a}_{s+1,2}$ & $\cdots$ & $0$\\
  \hline
    & $\tilde{b}_1$ & $\tilde{b}_2$ & $\cdots$ & $\tilde{b}_{s+1}$ \\
\end{tabular}
\end{minipage}
\end{center}

A critical aspect of our approach is the recognition that the mesh motion, governed by $\mathbf{F}(\mathbf{X}, \mathbf{Y})$, is not a source of numerical stiffness in our system. This is because the mesh velocity is typically of the same order of magnitude as the rigid-body velocity, and the mesh update equation itself does not contain the high-order spatial derivatives or the algebraic constraints (incompressibility, rigid-body motion) that are the primary contributors to stiffness in the FSI problem. Consequently, the characteristic time scales associated with mesh deformation are much less restrictive than those imposed by the fluid-structure interaction dynamics encapsulated in $\mathbf{G}(\mathbf{X}, \mathbf{Y})$. This physical insight allows us to select an IMEX partitioning of the Runge-Kutta method. Specifically, we treat the mesh equation for $\mathbf{X}$ as the non-stiff component and integrate it explicitly, while treating the FSI equation for $\mathbf{Y}$ as the stiff component and integrating it implicitly. This translates to choosing an explicit scheme for $\mathbf{F}$ (i.e., $a_{i,j} = 0$ for $j \geq i$) and a diagonally implicit scheme (DIRK) for $\mathbf{G}$ (i.e., $\tilde{a}_{i,j} = 0$ for $j > i$). This strategy effectively decouples the mesh evolution from the FSI solve within each time step, leading to a significant reduction in computational cost compared to a fully implicit approach, without compromising stability for the stiff components of the system. Following this rationale, we adopt the IMEX Runge-Kutta framework. 
\begin{remark}
For our simulations, we employ two specific schemes from \cite{ascher1997implicit}, represented by the following Butcher tableaux for the explicit (mesh motion) and implicit (FSI) parts, respectively:
\paragraph{Forward-Backward Euler:}
\begin{center}
\begin{minipage}{0.45\textwidth}
\centering
\textbf{Implicit part:}
\begin{tabular}{c|cc}
    0 & 0 & 0\\
  1 & 0 & 1\\
  \hline
    & 0 & 1 \\
\end{tabular}
\end{minipage}
\hfill
\begin{minipage}{0.45\textwidth}
\centering
\textbf{Explicit part:}
\begin{tabular}{c|cc}
  0 & 0 & 0\\
  1 & 1 & 0\\
  \hline
    & 1 & 0 \\
\end{tabular}
\end{minipage}
\end{center}
\paragraph{Two-stage second-order scheme:}
\begin{center}
    \begin{minipage}{0.45\textwidth}
\centering
\textbf{Implicit part:}
\begin{tabular}{c|ccc}
    0 & 0 & 0 & 0\\
  $\gamma$ & 0 & $\gamma$ & 0\\
  1 & 0 & $1-\gamma$ & $\gamma$\\
  \hline
    & 0 & $1-\gamma$ & $\gamma$ \\
\end{tabular}
\end{minipage}
\hfill
\begin{minipage}{0.45\textwidth}
\centering
\textbf{Explicit part:}
\begin{tabular}{c|ccc}
    0 & 0 & 0 & 0\\
  $\gamma$ & $\gamma$ & 0 & 0\\
  1 & $\delta$ & $1-\delta$ & 0\\
  \hline
    & $\delta$ & $1-\delta$ & 0\\
\end{tabular}
\end{minipage}
\end{center}
where $\gamma = 1 - \frac{1}{\sqrt{2}}$ and $\delta = 1 - \frac{1}{2\gamma}$.
\end{remark} 
\begin{remark}
    For more general case of Runge-Kutta, the matrix of $(a_{i,j})$ and $(\tilde{a}_{i,j})$ can be full matrices, leading to fully implicit schemes for both components and thus a better stability property. However, such schemes require solving a coupled nonlinear system involving all stages at each time step, which is computationally expensive. In contrast, the IMEX schemes we employ strike a balance between stability and efficiency by treating the stiff FSI dynamics implicitly while handling the non-stiff mesh motion explicitly.
\end{remark}

These schemes provide a balance between computational efficiency and temporal accuracy, making them well-suited for the coupled FSI-ALE problem at hand. While the traditional mathematical formulation of Runge-Kutta methods provides a clear theoretical framework through explicit slope terms and their combination via coefficient matrices $a_{ij}$, we adopt a reformulated expression for practical algorithm implementation. In our implementation, we substitute the slope terms $\mathbf{F}$ and $\mathbf{G}$ with some linear combinations of the solution variables at different stages. This approach simplifies the coding process and enhances computational efficiency, as it avoids the need to explicitly compute and store intermediate slope values. Instead, we directly manipulate the solution vectors, allowing for a more streamlined and memory-efficient implementation of the IMEX Runge-Kutta schemes.
Before we delve into the complete discretisation scheme, we give the fully discrete problem for the FSI system with a given domain mapping.
\begin{problem}\label{problem:general_IMEX}
Given $(\tilde{\mathbf{w}}, \mathcal{V}, \gamma, \mathcal{A}^t, \Delta t)$ representing an estimated previous velocity, mesh velocity, time scale parameter, and domain mapping, find the generalized solution
\[
{\Phi}_h =  (\tu_h, \tp_h, \tilde{\bU}_h, \tilde{\boldsymbol{\omega}}_h, \tilde{\boldsymbol{\lambda}}_h)  \in \tilde{V}_h \times \tilde{Q}_h \times \mathbb{R}^d \times \mathbb{R}^d \times \tilde{\boldsymbol{\Lambda}}_h
\]
such that
\begin{equation}\label{eq:general_IMEX}
\mathcal{F}\left( 
\frac{\tu_h - \tilde{\mathbf{w}}}{\gamma \Delta t} + (\tu_h - \mathcal{V}) \cdot \nabla_{\bx} \tu_h,\ 
{\Phi}_h,\ 
{\Psi}_h,\ 
\mathcal{A}^t 
\right) = 0
\end{equation}
for all test functions ${\Psi}_h$ in the same space.
\end{problem}

To facilitate the description and analysis of the time-stepping scheme, we introduce two abstract solution operators that encapsulate the core computational steps:

\begin{itemize}
    \item \textbf{Mesh Update Step} $\mathcal{M}$: 
    Given $(\mathbf{X}_h^{\rm prev}, \tU, \tomega, \alpha, \Delta\tau)$ representing the previous mesh displacement, rigid-body translational and rotational velocities, time scaling parameter, and time increment respectively, the operator $\mathcal{M}$ returns the updated mesh displacement $$\mathbf{X}_h = \mathcal{M}(\mathbf{X}_h^{\rm prev}, \tU, \tomega, \alpha, \Delta\tau)$$ by solving Problem~\ref{prob:mesh_update}. The corresponding domain mapping is then defined as: $\mathcal{A}^t(\mathbf{x}) = \mathbf{x} + \mathbf{X}_h(\mathbf{x})$.

    \item \textbf{FSI Solution Step} $\mathcal{S}$: Given $(\tilde{\mathbf{w}}, \mathcal{V}, \gamma, \mathcal{A}^t, \Delta t)$ representing an estimated previous velocity, mesh velocity, time scale parameter, domain mapping, and time increment respectively, the operator $\mathcal{S}$ returns the solution 
    \[
        {\Phi}_h = \mathcal{S}(\tilde{\mathbf{w}}, \mathcal{V}, \gamma, \mathcal{A}^t, \Delta t)
    \]
    by solving Problem~\ref{problem:general_IMEX}. This solution includes the fluid velocity, pressure, rigid-body velocities, and Lagrange multipliers.
\end{itemize}

These operators allow us to express the time advancement in a modular and stage-independent manner. Now we are ready to present the complete time-stepping schemes for the coupled FSI-ALE problem.

\paragraph{First-Order Semi-Implicit Scheme}
The first-order semi-implicit scheme proceeds as follows. At each time step $n$, given the solution $\boldsymbol{\Phi}_h^n = (\mathbf{u}_h^n, p_h^n, \mathbf{U}_h^n, \boldsymbol{\omega}_h^n, \boldsymbol{\lambda}_h^n)$ and the domain mapping $\mathcal{A}^{t_n}$, we compute the solution at $t_{n+1} = t_n + \Delta t$ in a single stage:

\begin{algorithm}
\renewcommand{\thealgocf}{PRK1}
\caption{First-Order Semi-Implicit Scheme}
\label{alg:first_order_FSI_ALE}
\SetAlgoLined
\DontPrintSemicolon
\KwIn{Initial values: $\mathbf{X}_h^0$, $\mathbf{u}_h^0$, $\mathbf{U}_h^0$, $\omega_h^0$, $\Delta t$}
\KwOut{Solutions at each time step: $\boldsymbol{\Phi}_h^n$, $\mathcal{A}^{t_n}$ for $n=0,1,2,\ldots$ until $t_n \geq T$}
\textbf{Initialization:}\;
$\alpha \leftarrow 1$\;
$\Delta\tau \leftarrow \Delta t$\;
\For{$n = 0, 1, 2, \ldots$ until $t_n \geq T$}{

    \textbf{Mesh update:}\;
    $\mathbf{X}_h^{n+1} \leftarrow \mathcal{M}(\mathbf{X}_h^n, \mathbf{U}_h^n, \omega_h^n, 1, \Delta t)$\;
    $\mathcal{A}^{t_{n+1}}(\mathbf{x}) \leftarrow \mathbf{x} + \mathbf{X}_h^{n+1}(\mathbf{x})$\;
    $\displaystyle \mathcal{V} \leftarrow \frac{\mathcal{A}^{t_{n+1}} - \mathcal{A}^{t_n}}{\Delta t}$\;

    \textbf{FSI solve:}\;
    $\tilde{\mathbf{w}} \leftarrow \mathbf{u}_h^n$\;
    $\boldsymbol{\Phi}_h^{n+1} \leftarrow \mathcal{S}(\mathbf{u}_h^n, \mathcal{V}, 1, \mathcal{A}^{t_{n+1}}, \Delta t)$\;
}

\Return{$\{\boldsymbol{\Phi}_h^n, \mathcal{A}^{t_n}\}_{n=0}^{N}$}

\end{algorithm}

\paragraph{Second-Order Semi-Implicit Scheme}
To achieve second-order temporal accuracy, we integrate the IMEX Runge-Kutta method into our FSI-ALE framework. The scheme consists of two stages per time step. We define the following coefficients for clarity:
\[
    \gamma = 1 - \frac{1}{\sqrt{2}}, \quad
    \delta = 1 - \frac{1}{2\gamma}, \quad
    \beta_0 = 1-\frac{1-\gamma}{\gamma}, \quad
    \beta_* = \frac{1-\gamma}{\gamma}, \quad
    c_0 = \delta, \quad
    c_* = 1-\delta.
\]
Given the solution $\boldsymbol{\Phi}_h^n$ and domain mapping $\mathcal{A}^{t_n}$ at time $t_n$, the update to $t_{n+1} = t_n + \Delta t$ proceeds as follows:

\begin{algorithm}
\renewcommand{\thealgocf}{PRK2}
\caption{Second-Order Semi-Implicit Scheme}
\label{alg:two_stage_FSI}
\SetAlgoLined
\DontPrintSemicolon
\KwIn{Initial values: $\mathbf{X}_h^0$, $\mathbf{u}_h^0$, $\mathbf{U}_h^0$, $\omega_h^0$, $\Delta t$}
\KwOut{Solutions at each time step: $\boldsymbol{\Phi}_h^n$, $\mathcal{A}^{t_n}$ for $n=0,1,2,\ldots$ until $t_n \geq T$}

\For{$n = 0, 1, 2, \ldots$ until $t_n \geq T$}{

\textbf{Stage 1: $t^* = t_n + \gamma \Delta t$}\;
\textbf{Mesh update:}\;
$\tilde{\mathbf{w}} \leftarrow \mathbf{u}_h^n$;
$\mathbf{X}_h^* \leftarrow \mathcal{M}(\mathbf{X}_h^n, \mathbf{U}_h^n, \omega_h^n, \gamma, \Delta t)$\;
$\mathcal{A}^{t^*}(\mathbf{x}) \leftarrow \mathbf{x} + \mathbf{X}_h^*(\mathbf{x})$\;
\textbf{FSI solve:}\;
$\displaystyle \mathcal{V}^* \leftarrow \frac{\mathcal{A}^{t^*} - \mathcal{A}^{t_n}}{\gamma \Delta t}$\;
$\boldsymbol{\Phi}_h^* \leftarrow \mathcal{S}(\mathbf{u}_h^n, \mathcal{V}^*, \gamma, \mathcal{A}^{t^*}, \Delta t)$\;

\textbf{Stage 2: $t_{n+1} = t_n + \Delta t$}\;
\textbf{Mesh update:}\;
$\mathbf{U}_h^{**} \leftarrow c_0 \mathbf{U}_h^n + c_* \mathbf{U}_h^*$\;
$\omega_h^{**} \leftarrow c_0 \omega_h^n + c_* \omega_h^*$\;
$\mathbf{X}_h^{n+1} \leftarrow \mathcal{M}(\mathbf{X}_h^n, \mathbf{U}_h^{**}, \omega_h^{**}, 1, \Delta t)$\;
$\mathcal{A}^{t_{n+1}}(\mathbf{x}) \leftarrow \mathbf{x} + \mathbf{X}_h^{n+1}(\mathbf{x})$\;
\textbf{FSI solve:}\;
$\displaystyle \mathcal{V}^{n+1} \leftarrow \frac{\mathcal{A}^{t_{n+1}} - \mathcal{A}^{t_n}}{\Delta t}$\;
$\tilde{\mathbf{w}} \leftarrow \beta_0 \mathbf{u}_h^n + \beta_* \mathbf{u}_h^*$\;
$\boldsymbol{\Phi}_h^{n+1} \leftarrow \mathcal{S}(\tilde{\mathbf{w}}, \mathcal{V}^{n+1}, \gamma, \mathcal{A}^{t_{n+1}}, \Delta t)$\;
}
\Return{$\{\boldsymbol{\Phi}_h^n, \mathcal{A}^{t_n}\}_{n=0}^{N}$}
\end{algorithm}

\begin{remark}
Although the algorithm we presented involves solving \eqref{eq:general_IMEX} in each time step which seems to be complicated and difficult for implementation, it can be simplified by using the relationship \eqref{eq:equivalence_coordinate}. The system \eqref{eq:general_IMEX} is equivalent to the system \eqref{eq:final_monolithic_system} with $\bx = \mathcal{A}^t(\tilde{\bx})$. Thus we can solve the system \eqref{eq:final_monolithic_system} in Euler coordinate and then pull back to the reference coordinate by using the mapping $\mathcal{A}^t$. The system \eqref{eq:final_monolithic_system} is a standard variational form and is much easier for numerical implementation.
\end{remark}

\begin{figure}[htbp]
    \centering
    \resizebox{0.8\textwidth}{!}{
        \begin{tikzpicture}[
    font=\sffamily,
    >=Stealth,
    box/.style={rectangle, draw, minimum width=3.2cm, minimum height=2.2cm, inner sep=0pt},
    label_node/.style={scale=1.2}
]

    \node (start_dots) at (1, 5.5) {$\cdots$};

    \node[label_node] (Th_k) at (0, 5) {$\mathcal{T}_h^{k}$};
    
    \node[box] (box1) at (4, 5.5) {};
    \draw (3.3,5.8) -- (3.5,6.0) -- (3.95,5.5) -- (3.5,5.4) -- cycle; 
    \draw (3.5,6.0) -- (3.5,5.4);
    \draw (4.3,5.5) circle (0.35cm); 
    \node[below=0.4cm of box1] {$\mathcal{T}_h^{k}$};

    \node[box] (box2) at (10, 5.5) {};
    \draw (9.1,5.7) -- (9.4,5.9) -- (10.25,5.5) -- (9.4,5.4) -- cycle; 
    \draw (9.4,5.9) -- (9.4,5.4);
    \draw (10.6,5.5) circle (0.35cm);
    \node[below=0.4cm of box2] {$\mathcal{T}_h^{k}$};

    \node[label_node] (Th_k_plus_1) at (0, 1.5) {$\mathcal{T}_h^{k+1}$};

    \node[box] (box3) at (10, 1.5) {};
    \draw (9.5,1.7) -- (9.8,1.9) -- (10.25,1.5) -- (9.8,1.4) -- cycle;
    \draw (9.8,1.9) -- (9.8,1.4);
    \draw (10.6,1.5) circle (0.35cm);
    \node[below=0.4cm of box3] {$\mathcal{T}_h^{k+1}$};

    \node[box] (box4) at (16, 1.5) {};
    \draw (15.4,1.7) -- (15.7,1.9) -- (16.55,1.3) -- (15.7,1.3) -- cycle;
    \draw (15.7,1.9) -- (15.7,1.3);
    \draw (16.9,1.3) circle (0.35cm);
    \node[below=0.4cm of box4] {$\mathcal{T}_h^{k+1}$};

    \draw[->, thick] (start_dots.east) -- (box1.west);
    \draw[->, thick] (Th_k) -- (Th_k_plus_1) node[midway, left] {remeshing};
    \draw[->, thick] (box1.east) -- (box2.west) node[midway, above] {time integration};
    \draw[->, thick] ([xshift=1.2cm]box2.south) -- ([xshift=1.2cm]box3.north) 
        node[midway, right=0.1cm, align=left] {
            if mesh quality is too low:\\
            \textbf{Remesh, Interpolation}
        };
    \draw[->, thick] (box3.east) -- (box4.west) node[midway, above] {time integration};
    \draw[->, thick] (box4.east) -- +(1, 0) node[right] {$\dots$};

    \draw[->, thick] (0, -1.5) -- (18, -1.5);
    \foreach \x/\t in {4/t^{n-1}, 10/t^{n}, 16/t^{n+1}} {
        \filldraw (\x, -1.5) circle (2.5pt) node[below=0.3cm, scale=1.3] {$\t$};
    }

\end{tikzpicture}
    }
    \caption{Mesh reconstruction and interpolation of the velocity field.}
    \label{fig:mesh_reconstruction}
\end{figure}
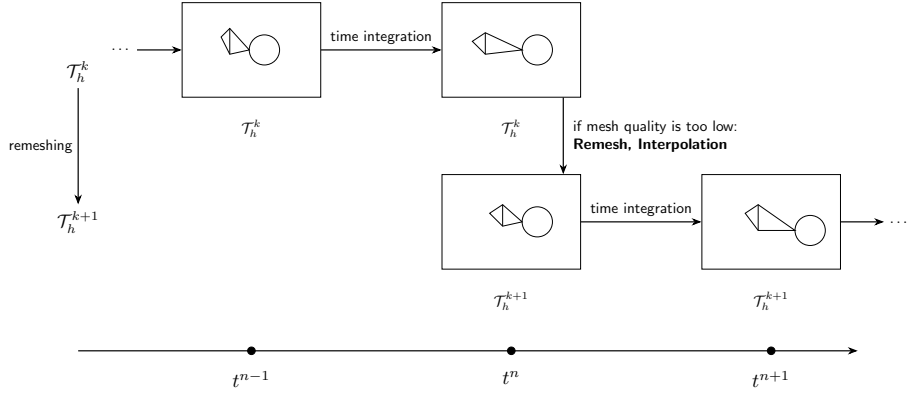

\begin{remark}
In the case of large displacement of the rigid particle, the mesh may be distorted. In this case, we need to reconstruct the mesh, and project the velocity field from the previous time step to the new mesh by interpolation or $L^2-$projection. In this paper, we use the interpolation for simplicity.
To the best of our knowledge, although there is no proof that the frequent interpolation of the velocity field will not affect the accuracy of the scheme, the numerical experiments still show that the algorithm produces good results as well \cite{gawlik2014high}. Thus a flow chat for the algorithm is illustrated in Figure \ref{fig:mesh_reconstruction}.
\end{remark}

\section{Well-posedness of the linearized system and the robust preconditioner}\label{sec:preconditioner}

In this section, we assume the ALE mapping $\mathcal{A}^t$ is given and analyze the well-posedness of the linearized system. The bilinear form $c^t(\cdot,\cdot)$ is defined as
\begin{equation}
    c^t(\boldsymbol{\lambda},\boldsymbol{\mu}) = \frac{\Reynolds}{\gamma\Delta t}(\boldsymbol{\lambda},\boldsymbol{\mu})_{\Omega_s^t} + 2(\varepsilon(\boldsymbol{\lambda}),\varepsilon(\boldsymbol{\mu}))_{\Omega_s^t}.
\end{equation}

We also discuss the preconditioner for the linearized system, which is crucial for the convergence of iterative solvers. To prepare for this discussion, we introduce the following notation for the function spaces:
\begin{equation}
    V = [H_0^1(\Omega)]^d, \quad
    Q = L^2_0(\Omega_f^t), \quad
    T_M = \mathbb{R}^d, \quad
    R_M = \mathbb{R}^d.
\end{equation}
where the inner products on these spaces are defined as:
\begin{equation}
    \begin{aligned}
        (\bu, \bv)_V = \frac{\Reynolds}{\gamma\Delta t}(\bu, \bv)_\Omega + 2(\varepsilon&(\bu), \varepsilon(\bv))_\Omega,~ 
        (p, q)_Q = \langle \tilde{S} p, q \rangle, \\
        (\bU, \mathbf{V})_{T_M} = c^t(\bU, \mathbf{V}), ~
        (\boldsymbol{\omega}, \boldsymbol{\xi})_{R_M} &= c^t(\boldsymbol{\omega} \times \mathbf{r}, \boldsymbol{\xi} \times \mathbf{r}), ~
        (\boldsymbol{\lambda}, \boldsymbol{\mu})_{\boldsymbol{\Lambda}} = c^t(\boldsymbol{\lambda}, \boldsymbol{\mu}),
    \end{aligned}
\end{equation}
with $\tilde{S}$ defined by 
\begin{equation}\label{eq:tildeS}
    \begin{aligned}
        &M: Q \to Q', \quad \langle Mp, q \rangle = (p, q)_{\Omega_f^t}, \\
        &A: Q \to Q', \quad \langle Ap, q \rangle = (\nabla p, \nabla q)_{\Omega_f^t}, \\
        &\tilde{S} = (M^{-1} + \frac{\Reynolds}{\gamma\Delta t} A^{-1})^{-1}.
    \end{aligned}
\end{equation}
The corresponding norms are given by:
\begin{equation}
    \begin{aligned}
        &\|\bu\|_V^2 = (\bu, \bu)_V, \quad
        \|p\|_Q^2 = (p, p)_Q, \quad
        \|\boldsymbol{\lambda}\|_{\boldsymbol{\Lambda}}^2 = (\boldsymbol{\lambda}, \boldsymbol{\lambda})_{\boldsymbol{\Lambda}}, \\
        &\|\bU\|_{T_M}^2 = (\bU, \bU)_{T_M}, \quad
        \|\boldsymbol{\omega}\|_{R_M}^2 = (\boldsymbol{\omega}, \boldsymbol{\omega})_{R_M}.
    \end{aligned}
\end{equation}

We neglect the convection term (since the Reynolds number is assumed to be small) and consider the following linearized system with augmented Lagrange multiplier:

\begin{problem}
Given ${\mathbf{f}}_\bv\in V'$, ${\mathbf{f}}_\mathbf{V}\in T_M'$, ${\mathbf{f}}_\bomega\in R_M'$, ${\mathbf{g}}\in Q'$, and ${\mathbf{h}}\in \boldsymbol{\Lambda}'$,
find $\bu\in V$, $p\in Q$, $\bU \in T_M$, $\boldsymbol{\omega}\in R_M$, ${\boldsymbol{\lambda}}\in \boldsymbol{\Lambda}$ such that
\begin{equation}\label{eq:linearised_continuous}
    \begin{aligned}
        &\frac{\Reynolds}{\gamma \Delta t}\, {m}^t({\bu}, {\bv}) + {a}^t(\bu,{\bv}) + {c}^t(\bu - {\bU} - \boldsymbol{\omega}\times \mathbf{r},\bv - {\mathbf{V}} - {\boldsymbol{\xi}}\times\mathbf{r}) \\
        &\quad + {b}^t(p, {\bv}) + {c}^t({\boldsymbol{\lambda}}, {\bv} - {\mathbf{V}} - {\boldsymbol{\xi}}\times\mathbf{r}) = \langle{\mathbf{f}}_\bv,{\bv}\rangle + \langle{\mathbf{f}}_\mathbf{V},{\mathbf{V}}\rangle + \langle{\mathbf{f}}_\bomega,{\boldsymbol{\xi}}\rangle, \\
        &{b}^t({q},\bu) = \langle{\mathbf{g}},{q}\rangle, \\
        &{c}^t({\boldsymbol{\mu}},\bu - \bU - \boldsymbol{\omega}\times\mathbf{r})  = \langle{\mathbf{h}},{\boldsymbol{\mu}}\rangle,
    \end{aligned}
\end{equation}
for all ${\bv}\in V$, ${q}\in Q$, ${\mathbf{V}} \in T_M$, ${\boldsymbol{\xi}}\in R_M$, ${\boldsymbol{\mu}}\in \boldsymbol{\Lambda}$, where $\gamma > 0$ is a parameter to be chosen for the scheme.
\end{problem}

The discretized linearized system is given by:

\begin{problem}
Given ${\mathbf{f}}_\bv\in V_h'$, ${\mathbf{f}}_\mathbf{V}\in T_M'$, ${\mathbf{f}}_\bomega\in R_M'$, ${\mathbf{g}}\in Q_h'$, and ${\mathbf{h}}\in \boldsymbol{\Lambda}_h'$,
find $\bu_h\in V_h$, $p_h\in Q_h$, $\bU_h \in T_M$, $\boldsymbol{\omega}_h\in R_M$, ${\boldsymbol{\lambda}}_h\in \boldsymbol{\Lambda}_h$ such that
\begin{equation}\label{eq:linearised_discretized}
    \begin{aligned}
        &\frac{\Reynolds}{\gamma \Delta t}\, {m}^t({\bu_h}, {\bv}_h) + {a}^t(\bu_h,{\bv}_h) + {c}^t(\bu_h - {\bU}_h - \boldsymbol{\omega}_h\times \mathbf{r},\bv_h - {\mathbf{V}}_h - {\boldsymbol{\xi}}_h\times\mathbf{r}) \\
        &\quad + {b}^t(p_h, {\bv}_h) + {c}^t({\boldsymbol{\lambda}}_h, {\bv}_h - {\mathbf{V}}_h - {\boldsymbol{\xi}}_h\times\mathbf{r}) = \langle{\mathbf{f}}_\bv,{\bv}_h\rangle + \langle{\mathbf{f}}_\mathbf{V},{\mathbf{V}}_h\rangle + \langle{\mathbf{f}}_\bomega,{\boldsymbol{\xi}}_h\rangle, \\
        &{b}^t({q}_h,\bu_h) = 0, \\
        &{c}^t({\boldsymbol{\mu}}_h,\bu_h - \bU_h - \boldsymbol{\omega}_h\times\mathbf{r})  = 0,
    \end{aligned}
\end{equation}
for all ${\bv}_h\in V_h$, ${q}_h\in Q_h$, ${\mathbf{V}}_h \in T_M$, ${\boldsymbol{\xi}}_h\in R_M$, ${\boldsymbol{\mu}}_h\in \boldsymbol{\Lambda}_h$, where $\gamma > 0$ is a parameter to be chosen for the scheme.
\end{problem}

\subsection{Main theorem}
The main results of this section are summarized in the following theorem, which establishes the coercivity, continuity, and LBB conditions necessary for the well-posedness of the linearized systems \eqref{eq:linearised_discretized} and convergence analysis of the preconditioner.
\begin{theorem}\label{lem:coercivity_continuity}
    There exist positive constants $\alpha$, $C_a$ and $C_b$ such that the following coercivity and continuity conditions hold for all $\bu \in V_h$, $\bU \in T_M$, $\boldsymbol{\omega} \in R_M$, $q \in Q_h$, and $\boldsymbol{\mu} \in \boldsymbol{\Lambda}_h$:
    \begin{equation}
        \begin{aligned}
            &\frac{\Reynolds}{\gamma\Delta t} m^t(\bu,\bu) + a^t(\bu,\bu) + c^t(\bu - \bU - \boldsymbol{\omega}\times\mathbf{r}, \bu - \bU - \boldsymbol{\omega}\times\mathbf{r}) \geq \alpha \left( \|\bu\|_V^2 + \|\bU\|_{T_M}^2 + \|\boldsymbol{\omega}\|_{R_M}^2 \right), \\
            &\frac{\Reynolds}{\gamma\Delta t} m^t(\bu,\bu) + a^t(\bu,\bu) + c^t(\bu - \bU - \boldsymbol{\omega}\times\mathbf{r}, \bu - \bU - \boldsymbol{\omega}\times\mathbf{r}) \leq C_a \left( \|\bu\|_V^2 + \|\bU\|_{T_M}^2 + \|\boldsymbol{\omega}\|_{R_M}^2 \right), \\
            &\left| b^t(q, \bu) + c^t(\boldsymbol{\mu}, \bu - \bU - \boldsymbol{\omega}\times\mathbf{r}) \right| \leq C_b \left( \|q\|_Q^2 + \|\boldsymbol{\mu}\|_{\boldsymbol{\Lambda}}^2 \right)^{1/2} \left( \|\bu\|_V^2 + \|\bU\|_{T_M}^2 + \|\boldsymbol{\omega}\|_{R_M}^2 \right)^{1/2}.
        \end{aligned}
    \end{equation}
    Moreover, the following LBB condition holds:
    \begin{equation}
        \inf_{\substack{q \in Q \\ \boldsymbol{\mu} \in \boldsymbol{\Lambda}}} \sup_{\substack{\bu \in V \\ \bU \in T_M \\ \boldsymbol{\omega} \in R_M}} \frac{b^t(q, \bu) + c^t(\boldsymbol{\mu}, \bu - \bU - \boldsymbol{\omega}\times\mathbf{r})}{\left( \|q\|_Q^2 + \|\boldsymbol{\mu}\|_{\boldsymbol{\Lambda}}^2 \right)^{1/2} \left( \|\bu\|_V^2 + \|\bU\|_{T_M}^2 + \|\boldsymbol{\omega}\|_{R_M}^2 \right)^{1/2}} = \beta > 0.
    \end{equation}
\end{theorem}
The proof of this theorem is given in the next subsection. Based on this theorem, we can establish the well-posedness of the linearized system \eqref{eq:linearised_discretized} by applying the LBB theory. Furthermore, these results provide a foundation for designing a robust preconditioner for the discrete system \eqref{eq:linearised_discretized}.

The following corollary states the well-posedness of the discrete system, which follows directly from Theorem \ref{lem:coercivity_continuity} and the LBB theory \cite{babuvska1973finite,boffi2013mixed}
\begin{corollary}
    The system \eqref{eq:linearised_discretized} is well-posed: the solution exists, is unique, and satisfies the stability estimate
    \begin{equation}
        \|\bu_h\|_V^2 + \|p_h\|_Q^2 + \|\bU_h\|_{T_M}^2 + \|\boldsymbol{\omega}_h\|_{R_M}^2 + \|\boldsymbol{\lambda}_h\|_{\boldsymbol{\Lambda}}^2 \leq C_s \left( \|{\mathbf{f}}_\bv\|_{V_h'}^2 + \|{\mathbf{f}}_\mathbf{V}\|_{T_M'}^2 + \|{\mathbf{f}}_\bomega\|_{R_M'}^2 + \|{\mathbf{g}}\|_{Q_h'}^2 + \|{\mathbf{h}}\|_{\boldsymbol{\Lambda}_h'}^2 \right),
    \end{equation}
    where the constant $C_s$ depends only on $\alpha$, $C_a$, $C_b$, and $\beta$ from Theorem \ref{lem:coercivity_continuity}.
\end{corollary}

Moreover, Theorem \ref{lem:coercivity_continuity} enables us to design a block preconditioner for the discrete system \eqref{eq:linearised_discretized}. The preconditioner is constructed based on the characterization of the Schur complements associated with the saddle-point structure of the system. Let $\mathbf{A}$ and $\mathbf{B}$ be the block matrices corresponding to the bilinear forms in \eqref{eq:linearised_discretized}, i.e.,
\begin{equation}
    \begin{aligned}
        \langle \mathbf{A} \begin{pmatrix} \bu_h \\ \bU_h \\ \boldsymbol{\omega}_h \end{pmatrix}, \begin{pmatrix} \bv_h \\ \mathbf{V}_h \\ \boldsymbol{\xi}_h \end{pmatrix} \rangle &= \frac{\Reynolds}{\gamma \Delta t} m^t({\bu_h}, {\bv}_h) + {a}^t(\bu_h,{\bv}_h) + {c}^t(\bu_h - {\bU}_h - \boldsymbol{\omega}_h\times \mathbf{r},\bv_h - {\mathbf{V}}_h - {\boldsymbol{\xi}}_h\times\mathbf{r}), \\
        \langle \mathbf{B}\begin{pmatrix} \bu_h \\ \bU_h \\ \boldsymbol{\omega}_h \end{pmatrix} ,\begin{pmatrix} q_h \\ \boldsymbol{\eta}_h \end{pmatrix}\rangle &= {b}^t(q_h, {\bu}_h) + {c}^t({\boldsymbol{\eta}}_h, {\bu}_h - {\bU}_h - {\boldsymbol{\omega}}_h\times\mathbf{r}).
    \end{aligned}
\end{equation}
Then the linearized system \eqref{eq:linearised_discretized} can be expressed in the block matrix form as
\begin{equation}
    \begin{pmatrix}
        \mathbf{A} & \mathbf{B}^T \\
        \mathbf{B} & 0
    \end{pmatrix}
    \begin{pmatrix}
        \bu_h \\ \bU_h \\ \boldsymbol{\omega}_h \\ p_h \\ \boldsymbol{\lambda}_h
    \end{pmatrix}
    =
    \begin{pmatrix}
        {\mathbf{f}}_\bv \\ {\mathbf{f}}_\mathbf{V} \\ {\mathbf{f}}_\bomega \\ {\mathbf{g}} \\ {\mathbf{h}}
    \end{pmatrix}.
\end{equation}
The Schur complement $\mathbf{S}$ associated with this system is given by $\mathbf{S} = \mathbf{B}\mathbf{A}^{-1}\mathbf{B}^T$.
We define the matrix $\tilde{\mathbf{S}}$ for preconditioning the Schur complement as
\begin{equation}
    \tilde{\mathbf{S}} = \begin{pmatrix}
        \tilde{S} & 0 \\
        0 & P_{\boldsymbol{\Lambda}}
    \end{pmatrix},
\end{equation}
where $\tilde{S}$ is defined in \eqref{eq:tildeS} and $P_{\boldsymbol{\Lambda}}: \boldsymbol{\Lambda}_h \to \boldsymbol{\Lambda}_h'$ is defined by
\begin{equation}
    \langle P_{\boldsymbol{\Lambda}} \boldsymbol{\mu}_h, \boldsymbol{\nu}_h \rangle = c^t(\boldsymbol{\mu}_h, \boldsymbol{\nu}_h), \quad \forall \boldsymbol{\mu}_h, \boldsymbol{\nu}_h \in \boldsymbol{\Lambda}_h.
\end{equation}
The proposed block preconditioner $\mathbf{P}$ for the discrete system \eqref{eq:linearised_discretized} is given by
\begin{equation}\label{eq:preconditioner}
    \mathbf{P} = 
    \begin{pmatrix}
        \mathbf{I} & -\mathbf{A}^{-1}\mathbf{B}^T \\
        0 & \mathbf{I}
    \end{pmatrix}
    \begin{pmatrix}
        \mathbf{A}^{-1} & 0 \\
        0 & \tilde{\mathbf{S}}^{-1}
    \end{pmatrix}
    \begin{pmatrix}
        \mathbf{I} & 0 \\
        -\mathbf{B}\mathbf{A}^{-1} & \mathbf{I}
    \end{pmatrix},
\end{equation}
where the inverse matrices $\mathbf{A}^{-1}$ and $\tilde{\mathbf{S}}^{-1}$ can be approximated using efficient solvers such as multigrid methods. The following corollary states that the condition number of the preconditioned Schur complement is bounded independently of the problem parameters, ensuring the robustness of the preconditioner.

\begin{corollary}
    The condition number of the preconditioned Schur complement $\tilde{\mathbf{S}}^{-1}\mathbf{S}$ is bounded independently of the problem parameters:
    \begin{equation}
        \kappa(\tilde{\mathbf{S}}^{-1}\mathbf{S}) \leq C,
    \end{equation}
    where the constant $C$ depends only on $\alpha$, $C_a$, $C_b$, and $\beta$ from Theorem \ref{lem:coercivity_continuity}.
\end{corollary}
The proof of this corollary follows directly from Theorem \ref{lem:coercivity_continuity} and the analysis in \cite{loghin2004analysis}.

\subsection{Proof of the main theorem}
First let us prove the coercivity and continuity conditions for the velocity space.
\begin{lemma}
    The coercivity condition and the continuity condition hold, i.e., there exist positive constants $\alpha$ and $C$ independent of $\Reynolds$, $\gamma$ and $\Delta t$ such that for all $\bu\in V$, $\bU\in T_M$, $\bomega\in R_M$, ${q}\in Q$ and ${\boldsymbol{\mu}}\in\boldsymbol{\Lambda}$, the following estimates hold
    \begin{equation}
        \begin{aligned}
            &\frac{\Reynolds}{\gamma\Delta t}{m}^t(\bu,\bu) + {a}^t(\bu,\bu) +  {c}^t(\bu-{\bU} - {\bomega}\times\mathbf{r},\bu-{\bU} - {\bomega}\times\mathbf{r})\geq \alpha (\|\bu\|^2_V  + \|\bU\|_{T_M}^2 + \|\bomega\|_{R_M}^2),\\
            &\frac{\Reynolds}{\gamma\Delta t}{m}^t(\bu,\bu) + {a}^t(\bu,\bu) +  {c}^t(\bu-{\bU} - {\bomega}\times\mathbf{r},\bu-{\bU} - {\bomega}\times\mathbf{r})\leq C_a (\|\bu\|^2_V  + \|\bU\|_{T_M}^2 + \|\bomega\|_{R_M}^2).
        \end{aligned}
    \end{equation}
\end{lemma}
\begin{proof}
    Notice that for translational and rotational motion are orthogonal with respect to the inner product $c^t(\cdot,\cdot)$, we have
\begin{equation}
    \begin{aligned}
        {c}^t(\bu-{\bU} - {\bomega}\times\mathbf{r},\bu-{\bU} - {\bomega}\times\mathbf{r}) &= {c}^t(\bu,\bu) + {c}^t({\bU},{\bU}) + {c}^t({\bomega}\times\mathbf{r},{\bomega}\times\mathbf{r}) - 2{c}^t(\bu,{\bU} + {\bomega}\times\mathbf{r})\\
        &\geq -\frac{1}{2}\|\bu\|_V^2 + \frac13\|{\bU}\|_{T_M}^2 + \frac{1}{3}\|{\bomega}\|_{R_M}^2,
    \end{aligned}
\end{equation}
where we used the Cauchy-Schwarz inequality
\begin{equation}
    2{c}^t(\bu,{\bU} + {\bomega}\times\mathbf{r})\leq\frac32 {c}^t(\bu,\bu) + \frac23 {c}^t({\bU} + {\bomega}\times\mathbf{r},{\bU} + {\bomega}\times\mathbf{r}).
\end{equation}
Thus we have the coercivity condition
\begin{equation}
    \begin{aligned}
        &\frac{\Reynolds}{\gamma\Delta t}{m}^t(\bu,\bu) + {a}^t(\bu,\bu) +  {c}^t(\bu-{\bU} - {\bomega}\times\mathbf{r},\bu-{\bU} - {\bomega}\times\mathbf{r})\geq \frac13 (\|\bu\|^2_V  + \|\bU\|_{T_M}^2 + \|\bomega\|_{R_M}^2).
    \end{aligned}
\end{equation}
The first continuity condition is straightforward to verify. We have
\begin{equation}
    \begin{aligned}
        \frac{\Reynolds}{\gamma\Delta t}{m}^t(\bu,\bu) + {a}^t(\bu,\bu) +  {c}^t(\bu-{\bU} - {\bomega}\times\mathbf{r},\bu-{\bU} - {\bomega}\times\mathbf{r})
        &\leq 3\|\bu\|^2_V + 2\|\bU\|_{T_M}^2 + 2\|\bomega\|_{R_M}^2 \\
        &\leq 3 \left( \|\bu\|^2_V + \|\bU\|_{T_M}^2 + \|\bomega\|_{R_M}^2 \right).
    \end{aligned}
\end{equation}
\end{proof}
The continuity condition for the bilinear forms $b^t(\cdot,\cdot)$ and $c^t(\cdot,\cdot)$ are more involved and requires additional analysis.
\begin{lemma}
    The continuity condition holds for the bilinear forms $b^t(\cdot,\cdot)$ and $c^t(\cdot,\cdot)$, i.e., there exists a positive constant $C$ independent of $\Reynolds$, $\gamma$ and $\Delta t$ such that for all $\bu\in V$, ${q}\in Q$, ${\boldsymbol{\mu}}\in\boldsymbol{\Lambda}$, the following estimate holds
    \begin{equation}
        {b}^t({q},\bu)+{c}^t({\boldsymbol{\mu}},\bu - {\bU} - {\bomega}\times\mathbf{r})\leq C(\|{q}\|_Q^2+\|{\boldsymbol{\mu}}\|_{\boldsymbol{\Lambda}}^2)^\frac12 (\|\bu\|_V^2 + \|{\bU}\|_{T_M}^2 + \|{\bomega}\|_{R_M}^2)^\frac12.
    \end{equation}
\end{lemma}
\begin{proof}
    The continuity condition for the bilinear form $c^t(\boldsymbol{\mu},\bu - \bU - \bomega\times\mathbf{r})$ is  straightforward. Indeed,
\begin{equation}
    \begin{aligned}
        {c}^t({\boldsymbol{\mu}},\bu - {\bU} - {\bomega}\times\mathbf{r})
        &\leq \|{\boldsymbol{\mu}}\|_{\boldsymbol{\Lambda}} \left( \|\bu\|_V + \|\bU\|_{T_M} + \|\bomega\|_{R_M} \right).
    \end{aligned}
\end{equation}
However, the continuity condition for the bilinear form $b^t(q,\bu)$ requires more detailed analysis. 
We define the linear operators $B: V \to Q'$, $M_\bv: V \to V'$, and $A_\bv: V \to V'$ by
\begin{equation}
    \begin{aligned}
        &\langle B\bu, q \rangle = b^t(q, \bu), \\
        &\langle A_\bv\bu, \bv \rangle = a^t(\bu, \bv), \\
        &\langle M_\bv\bu, \bv \rangle = m^t(\bu, \bv),
    \end{aligned}
\end{equation}
and let $B^*: Q \to V'$ be the adjoint of $B$, and $T_\bv = \frac{\Reynolds}{\gamma\Delta t} M_\bv + A_\bv$. Since $A_\bv$ and $M_\bv$ are self-adjoint and positive definite, we obtain the identity
\begin{equation}
    \begin{aligned}
    \left( \sup_{\bu \in V} \frac{b^t(q, \bu)}{\|\bu\|_V} \right)^2 
    &= \sup_{\bu \in V} \frac{\langle B^* q, \bu \rangle^2}{\langle T_\bv \bu, \bu \rangle} 
     = \sup_{\bu \in V} \frac{\langle T_\bv^{-1} B^* q, T_\bv \bu \rangle^2}{\langle T_\bv^{-1} T_\bv \bu, T_\bv \bu \rangle} 
     = \langle T_\bv^{-1} B^* q, B^* q \rangle 
     = \langle B T_\bv^{-1} B^* q, q \rangle.
    \end{aligned}
\end{equation}
Denoting $S = B T_\bv^{-1} B^*$, it suffices to show that
\begin{equation}
    \langle S q, q \rangle \leq C \langle \tilde{S} q, q \rangle = C \|q\|_Q^2
\end{equation}
for all $q \in Q$. This is equivalent to
\begin{equation}
    \langle S \tilde{S}^{-1} q', \tilde{S}^{-1} q' \rangle \leq C \langle \tilde{S}^{-1} q', q' \rangle
\end{equation}
for all $q' \in Q'$. It therefore suffices to show that the spectral radius of $S \tilde{S}^{-1}$ is bounded by $C$.
It is easy to show that $S \tilde{S}^{-1}$ and $\tilde{S}^{-1} S$ have the same spectrum.
As the closure of the numerical range of a self-adjoint operator with respect to an inner product defined by $\langle \cdot, S\cdot \rangle$ is the convex hull of its spectrum, it suffices to show that
\begin{equation}
    \langle \tilde{S}^{-1} S q, S q \rangle \leq C \langle S q, q \rangle
\end{equation}
for all $q \in Q$. By the definition of $\tilde{S}$ in \eqref{eq:tildeS}, we have
\begin{equation}
    \begin{aligned}
        \langle \tilde{S}^{-1} S q, S q \rangle = \frac{\Reynolds}{\gamma\Delta t} \langle A^{-1} S q, S q \rangle + \langle M^{-1} S q, S q \rangle.
    \end{aligned}
\end{equation}
For the first term, we have
\begin{equation}
    \begin{aligned}
        \frac{\Reynolds}{\gamma\Delta t} \langle A^{-1} S q, S q \rangle 
        &= \frac{\Reynolds}{\gamma\Delta t} \sup_{p \in Q} \frac{\langle S q, p \rangle^2}{\langle A p, p \rangle} 
         = \sup_{p \in Q} \frac{\langle T_\bv^{-1} B^* q, B^* p \rangle^2}{(\nabla p, \nabla p)_{\Omega_f^t}} 
         = \frac{\Reynolds}{\gamma\Delta t} \sup_{p \in Q} \frac{(T_\bv^{-1} B^* q, \nabla p)^2}{(\nabla p, \nabla p)_{\Omega_f^t}} \\
        &= \frac{\Reynolds}{\gamma\Delta t} \|T_\bv^{-1} B^* q\|_{L^2(\Omega_f^t)}^2 
         = \left\langle \frac{\Reynolds}{\gamma\Delta t} M_\bv T_\bv^{-1} B^* q, T_\bv^{-1} B^* q \right\rangle \\
        &\leq \left\langle \left( \frac{\Reynolds}{\gamma\Delta t} M_\bv + A_\bv \right) T_\bv^{-1} B^* q, T_\bv^{-1} B^* q \right\rangle 
         = \langle S q, q \rangle.
    \end{aligned}
\end{equation}
For the second term, we claim that
\begin{equation}
    \begin{aligned}
        \langle M^{-1} S q, S q \rangle \leq \langle S q, q \rangle.
    \end{aligned}
\end{equation}
To prove this, we again apply the argument on the spectral radius of a self-adjoint operator, so it suffices to show that
\begin{equation}
    \langle M^{-1} S q, M q \rangle \leq \langle M q, q \rangle,
\end{equation}
i.e.,
\begin{equation} 
    \langle S q, q \rangle \leq \langle M q, q \rangle.
\end{equation}
We have
\begin{equation}
    \begin{aligned}
        \langle S q, q \rangle 
        &= \langle T_\bv^{-1} B^* q, B^* q \rangle 
         = \langle T_\bv^{-1} B^* q, T_\bv T_\bv^{-1} B^* q \rangle \\
        &= \sup_{\bu \in V} \frac{\langle T_\bv^{-1} B^* q, T_\bv \bu \rangle^2}{\langle T_\bv \bu, \bu \rangle} 
         \leq C \sup_{\bu \in V} \frac{\langle B \bu, q \rangle^2}{\|\bu\|_{H^1(\Omega_f^t)}^2} \\
        &\leq C \langle M q, q \rangle,
    \end{aligned} 
\end{equation}
where we used $\|\nabla \cdot \bu| \leq C \|\bu\|_{H^1(\Omega_f^t)}$ and Korn's inequality. Combining both terms, we obtain
\begin{equation}
    \langle \tilde{S}^{-1} S q, S q \rangle \leq C \langle S q, q \rangle.
\end{equation}
This implies that the spectral radius of $S \tilde{S}^{-1}$ is bounded by $C$, and hence
\begin{equation}
    \sup_{\bu \in V} \frac{b^t(q, \bu)}{\|\bu\|_V} \leq C \|q\|_Q.
\end{equation}
Thus, the continuity condition for the bilinear form $b^t(q, \bu)$ is established.
\end{proof}
\begin{remark}
    The conclusion above holds for the discretized system \eqref{eq:linearised_discretized} as well, since the finite element spaces $V_h$, $Q_h$, $T_M$, $R_M$ and $\boldsymbol{\Lambda}_h$ are all subspaces of the continuous spaces $V$, $Q$, $T_M$, $R_M$ and $\boldsymbol{\Lambda}$, respectively. The bilinear forms are also the same as the continuous case. Thus we have the same coercivity and continuity conditions for the discretized system \eqref{eq:linearised_discretized}.
\end{remark}

The proof of the LBB condition requires some auxiliary results for the Stokes equation.
First we make the following regularity assumption for Stokes equation.
\begin{assumption}\label{assumption:regularity}
    We assume that if $(\bu,p)$ is the solution of the Stokes equation
    \begin{equation}\label{eq:stokes_aux1}
        \begin{aligned}
        -\Delta \bu+ \nabla p &= \mathbf{f},&\text{in }\Omega_f^t,\\
        \nabla \cdot \bu &= 0, &\text{in }\Omega_f^t,\\
        \bu &= 0, &\text{on }\partial\Omega_f^t,\\
        \int_{\Omega_f^t} p \ dx &= 0,
    \end{aligned}
    \end{equation}
    then 
    \begin{equation}
        \|\bu\|_{H^2(\Omega)} + \|\nabla p\|_{L^2(\Omega)}\leqslant C \|\mathbf{f}\|_{L^2(\Omega)}.
    \end{equation}
\end{assumption}
This gives the following result which we will use in the analysis of the LBB condition.
\begin{lemma}\label{lemma:stokes_estimate}
    Let $(\bu,p)$ be the solution of the Stokes equation
    \begin{equation}\label{eq:stokes_aux2}
    \begin{aligned}
        -\Delta \bu+ \nabla p &= 0,&\text{in }\Omega_f^t,\\
        \nabla \cdot \bu &= g, &\text{in }\Omega_f^t,\\
        \bu &= 0, &\text{on }\partial\Omega_f^t,\\
        \int_{\Omega_f^t} p \ dx &= 0,
    \end{aligned}
    \end{equation}
    with $g \in L_0^2(\Omega_f^t)$. Under Assumption~\ref{assumption:regularity}, we have 
    \begin{equation}
        \|\bu\|_{L^2(\Omega_f^t)} \leq C \|g\|_{H^{-1}(\Omega_f^t)}.
    \end{equation}
\end{lemma}
\begin{proof}
    Let $(\Psi, \theta)$ be the solution of the dual problem:
    \begin{equation}
        -\Delta \Psi + \nabla \theta = \bu, \quad \nabla \cdot \Psi = 0, \quad \Psi|_{\partial\Omega_f^t} = 0, \quad \int \theta = 0.
    \end{equation}
    By Assumption~\ref{assumption:regularity}, we have
    \begin{equation}
        \|\Psi\|_{H^2(\Omega_f^t)} + \|\nabla \theta\|_{L^2(\Omega_f^t)} \leq C \|\bu\|_{L^2(\Omega_f^t)}.
    \end{equation}
    Testing the equation for $\bu$ with $\Psi$ and using $\nabla \cdot \Psi = 0$, we obtain
    \begin{equation}
        \int_{\Omega_f^t} \nabla \bu : \nabla \Psi \, dx = 0.
    \end{equation}
    Testing the dual equation with $\bu$, we get
    \begin{equation}
        \|\bu\|_{L^2}^2 = \int_{\Omega_f^t} \nabla \Psi : \nabla \bu \, dx - \int_{\Omega_f^t} \theta g \, dx = - \int_{\Omega_f^t} \theta g \, dx.
    \end{equation}
    Thus,
    \begin{equation}
        \|\bu\|_{L^2}^2 \leq \|\theta\|_{H^1} \|g\|_{H^{-1}} \leq C \|\bu\|_{L^2} \|g\|_{H^{-1}},
    \end{equation}
    which implies the desired estimate.
\end{proof}
We also recall the standard inf-sup condition for the Stokes equation.
\begin{lemma}\label{lemma:infsup_stationary_stokes}
    There exists a constant $C$ satisfying 
    \begin{equation}\label{eq:infsup}
        \|p\|_{L^2_0(\Omega_f^t)}\leq C\sup_{\bu\in H_0^1(\Omega_f^t)}\frac{(\nabla\cdot\bu,p)_{\Omega_f^t}}{\|\bu\|_{H^1(\Omega_f^t)}}
    \end{equation}
\end{lemma}
This is a standard result for the Stokes equation, and we refer to \cite{girault1979finite} for the proof. This classical result ensures the well-posedness of the Stokes problem and will play a crucial role in our analysis. To establish the LBB condition for the coupled system~\eqref{eq:linearised_continuous}, we follow a constructive approach: for given $q \in Q$ and $\boldsymbol{\mu} \in \boldsymbol{\Lambda}$, we aim to construct a velocity field $\bu \in V$ such that the combined bilinear forms $b^t(q,\bu)$ and $c^t(\boldsymbol{\mu},\bu - \bU - \bomega \times \mathbf{r})$ are simultaneously bounded from below, while $\bu$ remains uniformly bounded in the $V$-norm.

The construction proceeds in two steps. First, we choose $\bu_1$ to match $\boldsymbol{\mu}$ in the solid region $\Omega_s^t$, ensuring a lower bound for the interface term. Second, we correct the divergence error introduced by $\bu_1$ by solving an auxiliary Stokes problem, using the inf-sup condition \eqref{eq:infsup} to control the pressure term $b^t(q,\bu)$. This strategy allows us to decouple the treatment of the Lagrange multiplier and the pressure, while maintaining uniform stability.

Now we can prove the LBB condition for the system \eqref{eq:linearised_continuous}.
\begin{theorem}
    Assume that the Assumption \ref{assumption:regularity} holds, then the LBB condition holds for the system \eqref{eq:linearised_continuous} , i.e., there exists a constant $\beta > 0$ independent of $\Reynolds$, $\gamma$ and $\Delta t$ such that
    \begin{equation}
        \inf_{{q}\in Q, \boldsymbol{\mu}\in\boldsymbol{\Lambda}} \sup_{\bu\in V,{\bU}\in T_M,{\bomega}\in R_M} \frac{{b}^t({q},\bu)+{c}^t({\boldsymbol{\mu}},\bu - {\bU} - {\bomega}\times\mathbf{r})}{(\|{q}\|_Q^2\|{\boldsymbol{\mu}}\|_{\boldsymbol{\Lambda}}^2)^\frac12 (\|\bu\|_V^2 + \|{\bU}\|_{T_M}^2 + \|{\bomega}\|_{R_M}^2)^\frac12} = \beta > 0.
    \end{equation}
\end{theorem}
\begin{proof}
    It suffices to show that there exists a constant $C$ such that for any ${q}\in Q$, ${\boldsymbol{\mu}}\in \boldsymbol{\Lambda}$, there exists $\bu\in V$, ${\bU}\in T_M$, ${\bomega}\in R_M$ such that
    \begin{equation}
        \begin{aligned}
        &{b}^t({q},\bu)+{c}^t({\boldsymbol{\mu}},\bu - {\bU} - {\bomega}\times\mathbf{r}) \geq \|{q}\|_Q^2+\|{\boldsymbol{\mu}}\|_{\boldsymbol{\Lambda}}^2,\\
        &\|\bu\|_V^2 + \|{\bU}\|_{T_M}^2 + \|{\bomega}\|_{R_M}^2 \leq C(\|{q}\|_Q^2 + \|{\boldsymbol{\mu}}\|_{\boldsymbol{\Lambda}}^2).
        \end{aligned}
    \end{equation}
    First we choose $\bu_1 \in V$ such that 
    \begin{equation} 
        \bu_1 = \boldsymbol{\mu}, \text{ in } \Omega_s^t, \quad \|\bu_1\|_V^2 \leq C\|\boldsymbol{\mu}\|_{\boldsymbol{\Lambda}}^2.
    \end{equation}
    This can be done by extending the function $\boldsymbol{\mu}$ to the whole domain $\Omega$ by constructing an extension operator $E:\boldsymbol{\Lambda}\to V$ such that 
    \begin{equation}
        \begin{aligned}
            &E \boldsymbol{\mu} = \boldsymbol{\mu}, \text{ in }\Omega_s^t,\\
            &\|E\boldsymbol{\mu}\|_{L^2(\Omega)}\leq C \|\boldsymbol{\mu}\|_{L^2(\Omega_s^t)},\\
            &\|E\boldsymbol{\mu}\|_{H^1(\Omega)}\leq C\|\boldsymbol{\mu}\|_{H^1(\Omega_s^t)}.
        \end{aligned}
    \end{equation} We refer to \cite{evans2022partial} for the construction of such operator. The remaining part is to choose $\bu_2\in V$ such that
    \begin{equation}
        \begin{aligned}\label{cond:bu2}
            &(\nabla\cdot(\bu_1+\bu_2),q)_{\Omega_f^t} = \|q\|_Q^2, \text{ in } \Omega_f^t,\\
            &\bu_2 = 0, \text{ in } \Omega_s^t,\\
            &\|\bu_2\|_V^2 \leq C(\|{q}\|_Q^2 + \|{\boldsymbol{\mu}}\|_{\boldsymbol{\Lambda}}^2).
        \end{aligned}
    \end{equation}
    To construct such a function, we consider the stationary Stokes equation. Let $(\bu_2,p)$ be the solution of the Stokes equation
    \begin{equation}\label{eq:stokes_aux3}
        \begin{aligned}
            -\Delta \bu_2 + \nabla p &= 0, &\text{in }\Omega_f^t,\\
            \nabla\cdot\bu_2 &= \tilde{S}q - \nabla\cdot\bu_1, &\text{in }\Omega_f^t,\\
            \bu_2 &= 0, &\text{on }\partial\Omega_f^t,\\
            \int_{\Omega_f^t} p \ dx &= 0.
        \end{aligned}
    \end{equation}
    We will show that such a function exists and satisfies the condition \eqref{cond:bu2}. 
    The first condition is trivial by multiplying the second equation in \eqref{eq:stokes_aux3} by $q$ and integrating over $\Omega_f^t$. For the third condition, we expand the norm $\|\bu_2\|_V^2$ as
    \begin{equation}
        \begin{aligned}
            \|\bu_2\|_V^2 &= \frac{\Reynolds}{\gamma\Delta t}\|\bu_2\|_{L^2(\Omega_f^t)}^2 + |\bu_2|_{H^1(\Omega_f^t)}^2.
        \end{aligned}
    \end{equation}
    For the first term, we can use the result from lemma \ref{lemma:stokes_estimate} to get
    \begin{equation}
        \frac{\Reynolds}{\gamma\Delta t}\|\bu_2\|_{L^2(\Omega_f^t)}^2\leq C \frac{\Reynolds}{\gamma \Delta t}\|\tilde{S}q-\nabla\cdot \bu_1\|_{H^{-1}(\Omega_f^t)}^2.
    \end{equation}
    For the second term, we use lemma \ref{lemma:infsup_stationary_stokes} to get
    \begin{equation}
        |\bu_2|_{H^1(\Omega_f^t)}^2\leq C\|\tilde{S}q - \nabla\cdot\bu_1\|_{L^2(\Omega_f^t)}^2.
    \end{equation}
    Notice that 
    \begin{equation}
        \begin{aligned}
            \frac{\Reynolds}{\gamma\Delta t}\|\tilde{S}q\|_{H^{-1}(\Omega_f^t)}^2 + \|\tilde{S}q\|_{L^2(\Omega_f^t)}^2 &= \frac{\Reynolds}{\gamma\Delta t}(\sup_{ p\in Q}\frac{\langle\tilde{S}q,p\rangle}{|p|_{H^1(\Omega_f^t)}})^2 + \langle M^{-1}\tilde{S}q,\tilde{S}q\rangle\\
            & = \frac{\Reynolds}{\gamma\Delta t}\sup_{ p\in Q}\frac{(\nabla A^{-1}\tilde{S}q,\nabla A^{-1}\tilde{S}p)_{\Omega_f^t}^2}{(\nabla A^{-1}\tilde{S}p,\nabla A^{-1}\tilde{S}p)_{\Omega_f^t}} + \langle M^{-1}\tilde{S}q,\tilde{S}q\rangle\\
            & = \frac{\Reynolds}{\gamma\Delta t}\frac{(\nabla A^{-1}\tilde{S}q,\nabla  A^{-1}\tilde{S}q)_{\Omega_f^t}^2}{(\nabla  A^{-1}\tilde{S}q,\nabla  A^{-1}\tilde{S}q)_{\Omega_f^t}} + \langle M^{-1}\tilde{S}q,\tilde{S}q\rangle\\
            &= \frac{\Reynolds}{\gamma\Delta t}(\nabla  A^{-1}\tilde{S}q,\nabla  A^{-1}\tilde{S}q)_{\Omega_f^t} + \langle M^{-1}\tilde{S}q,\tilde{S}q\rangle\\
            &= \frac{\Reynolds}{\gamma\Delta t}\langle A^{-1}\tilde{S}q, \tilde{S}q\rangle + \langle M^{-1}\tilde{S}q,\tilde{S}q\rangle\\
            &= \langle \tilde{S}q,q\rangle = \|q\|_Q^2.
        \end{aligned}
    \end{equation}
    Moreover, the following estimate holds from Poincar\'e inequality
    \begin{equation}
        \begin{aligned}
        \frac{\Reynolds}{\gamma\Delta t}\|\nabla\cdot \bu_1\|_{H^{-1}(\Omega_f^t)}^2 + \|\nabla\cdot \bu_1\|_{L^2(\Omega_f^t)}^2 &\leq C(\frac{\Reynolds}{\gamma\Delta t}\|\bu_1\|_{L^2(\Omega_f^t)}^2 + \|\bu_1\|_{H^1(\Omega_f^t)}^2)\\
        &\leq C\|\bu_1\|_V^2\leq C\|\boldsymbol{\mu}\|_{\boldsymbol{\Lambda}}^2.
        \end{aligned}
    \end{equation}
    Combining the above inequalities, we have
    \begin{equation}
        \|\bu_2\|_V^2 \leq C (\|{q}\|_Q^2 + \|{\boldsymbol{\mu}}\|_{\boldsymbol{\Lambda}}^2).
    \end{equation}
    Finally, we can choose $\bu = \bu_1 + \bu_2$, ${\bU} = \mathbf{0}$, ${\bomega} = \mathbf{0}$, and we have
    \begin{equation}
        \begin{aligned}
            &{b}^t({q},\bu)+{c}^t({\boldsymbol{\mu}},\bu - {\bU} - {\bomega}\times\mathbf{r}) \geq \|{q}\|_Q^2\|{\boldsymbol{\mu}}\|_{\boldsymbol{\Lambda}}^2,\\
            &\|\bu\|_V^2 + \|{\bU}\|_{T_M}^2 + \|{\bomega}\|_{R_M}^2 \leq C(\|{q}\|_Q^2 + \|{\boldsymbol{\mu}}\|_{\boldsymbol{\Lambda}}^2).
        \end{aligned}
    \end{equation}
    Thus we have proved the LBB condition for the system \eqref{eq:linearised_continuous}.
\end{proof}
We now turn to the discretized system~\eqref{eq:linearised_discretized}. 
The coercivity and continuity conditions for the discrete problem follow immediately from the continuous case, as the finite element spaces $V_h$, $Q_h$, $T_M$, $R_M$, and $\boldsymbol{\Lambda}_h$ are subspaces of their continuous counterparts, and the bilinear forms retain the same structure. 

However, the discrete LBB condition requires a more refined argument. 
For the discrete LBB condition, however, a constructive argument is required to ensure uniform stability. 
We thus assume the following mesh regularity.

\begin{assumption}\label{assumption:mesh}
    We assume that the family of triangulations $(\mathcal{T}_h)_h$ is shape regular, i.e., there exists a constant $\sigma > 0$ such that for all $h$ and all simplices $K \in \mathcal{T}_h$,
    \[
    \frac{h_K}{\rho_K} \leq \sigma,
    \]
    where $h_K$ is the diameter of $K$ and $\rho_K$ is the diameter of the largest ball contained in $K$. 
\end{assumption}

This condition ensures the existence of a bounded Scott--Zhang interpolation operator $\Pi_h: H^1(\Omega) \to V_h$ satisfying
\[
\|\Pi_h v\|_{L^2(\Omega)} \leq C \|v\|_{L^2(\Omega)}, \quad
\|\Pi_h v\|_{H^1(\Omega)} \leq C \|v\|_{H^1(\Omega)},
\]
for all $v \in H^1(\Omega)$, with $C$ depending only on $\sigma$ and the spatial dimension, but independent of $h$. For further details, we refer to~\cite{scott1990finite,brenner2008mathematical}.

With this tool at our disposal, we are now able to construct the required discrete test functions in a stable manner, which leads to the following discrete LBB condition.
\begin{lemma}\label{thm:lbb_discrete}
    Assume that Assumptions~\ref{assumption:regularity} and~\ref{assumption:mesh} hold.
    Then the LBB condition holds for the system~\eqref{eq:linearised_discretized}, i.e., there exists a constant $\beta > 0$, independent of $h$, $\Reynolds$, $\gamma$, and $\Delta t$, such that
    \begin{equation}
        \inf_{{q}_h\in Q_h, \boldsymbol{\mu}_h\in \boldsymbol{\Lambda}_h} \sup_{\bu_h\in V_h,{\bU}_h\in T_M,{\bomega}_h\in R_M} \frac{{b}^t({q}_h,\bu_h)+{c}^t({\boldsymbol{\mu}}_h,\bu_h - {\bU}_h - {\bomega}_h\times\mathbf{r})}{(\|{q}_h\|_Q^2+\|{\boldsymbol{\mu}}_h\|_{\boldsymbol{\Lambda}}^2)^\frac12 (\|\bu_h\|_V^2 + \|{\bU}_h\|_{T_M}^2 + \|{\bomega}_h\|_{R_M}^2)^\frac12} = \beta > 0.
    \end{equation}
\end{lemma}

\begin{proof}
    The proof is analogous to the continuous case. We first choose $\bu_{h,1} \in V_h$ such that 
    \begin{equation} 
        \bu_{h,1} = \boldsymbol{\mu}_h \quad \text{in } \Omega_s^t, \quad \|\bu_{h,1}\|_V^2 \leq C\|\boldsymbol{\mu}_h\|_{\boldsymbol{\Lambda}}^2.
    \end{equation}
    This is achieved by interpolating the function $\bu_1$ into the finite element space $V_h$. Let $\Pi_h$ denote the Scott--Zhang interpolation operator. Then we have
    \begin{equation}
        \begin{aligned}
            &\frac{\Reynolds}{\gamma \Delta t}\|\Pi_h\bu_1\|_{L^2(\Omega_f^t)}^2 \leq C  \frac{\Reynolds}{\gamma \Delta t}\|\bu_1\|_{L^2(\Omega_f^t)}^2\leq C \|\boldsymbol{\mu}_h\|_{\boldsymbol{\Lambda}}^2,\\
            &|\Pi_h\bu_1|_{H^1(\Omega_f^t)}^2 \leq C |\bu_1|_{H^1(\Omega_f^t)}^2 \leq C \|\boldsymbol{\mu}_h\|_{\boldsymbol{\Lambda}}^2.
        \end{aligned}
    \end{equation}
    Thus, $\bu_{h,1} = \Pi_h\bu_1$ satisfies the required condition. Next, we construct $\bu_{h,2}\in V_h$ such that
    \begin{equation}
        \begin{aligned}\label{cond:bu2_discretized}
            &(\nabla\cdot(\bu_{h,1}+\bu_{h,2}),q_h)_{\Omega_f^t} = \|q_h\|_Q^2, \quad \text{in } \Omega_f^t,\\
            &\bu_{h,2} = 0, \quad \text{in } \Omega_s^t,\\
            &\|\bu_{h,2}\|_V^2 \leq C(\|{q}_h\|_Q^2 + \|{\boldsymbol{\mu}}_h\|_{\boldsymbol{\Lambda}}^2).
        \end{aligned}
    \end{equation}
    This is accomplished by solving a discrete Stokes problem. Let $(\bu_{h,2},p_h)$ satisfy
    \begin{equation}\label{eq:stokes_aux4}
        \begin{aligned}
            (\nabla\bu_{h,2},\nabla\Psi_h) - (\nabla\cdot\Psi_h,p_h) &= 0, &\forall \Psi_h \in V_{h,0}^f,\\
            (\nabla\cdot\bu_{h,2},\phi_h) &= \langle \tilde{S}q_h, \phi_h\rangle - (\nabla\cdot \bu_{h,1}, \phi_h), &\forall \phi_h \in Q_h^f,\\
            \bu_{h,2} &= 0, &\text{in }\partial\Omega\cup \Omega_s^t,\\
            \int_{\Omega_f^t} p_h \ dx &= 0,
        \end{aligned}
    \end{equation}
    where $V_{h,0}^f$ denotes the finite element space with zero boundary conditions on $\partial\Omega_f^t$. Let $(\bu_2,p)$ be the solution of the continuous Stokes problem
    \begin{equation}\label{eq:stokes_aux5}
        \begin{aligned}
            -\Delta \bu_2 + \nabla p &= 0, &\text{in }\Omega_f^t,\\
            \nabla\cdot\bu_2 &= \tilde{S}q_h - \nabla\cdot\bu_{h,1}, &\text{in }\Omega_f^t,\\
            \bu_2 &= 0, &\text{on }\partial\Omega_f^t,\\
            \int_{\Omega_f^t} p \ dx &= 0.
        \end{aligned}
    \end{equation}
    Then, by \eqref{eq:stokes_aux4} and \eqref{eq:stokes_aux5} we have 
    \begin{equation}
        \begin{aligned}
            |\bu_{h,2}|_{H^1(\Omega_f^t)}^2 &= (\nabla\cdot\bu_{h,2},p_h)_{\Omega_f^t} = (\nabla\cdot\bu_{2},p_h)_{\Omega_f^t}\leq \|p_h\|_{L^2(\Omega_f^t)}|\bu_{2}|_{H^{1}(\Omega_f^t)}.
        \end{aligned}
    \end{equation}
    For $\|p_h\|_{L^2(\Omega_f^t)}$, we apply the discrete inf-sup condition for the Stokes equation to obtain
    \begin{equation}
        \|p_h\|_{L^2(\Omega_f^t)}\leq C\sup_{\mathbf{w}_h\in V_{h,0}^f}\frac{(\nabla\cdot\mathbf{w}_h,p_h)_{\Omega_f^t}}{|\mathbf{w}_h|_{H^1(\Omega_f^t)}} = C\sup_{\mathbf{w}_h\in V_{h,0}^f}\frac{(\nabla\mathbf{w}_h,\nabla\bu_{h,2})_{\Omega_f^t}}{|\mathbf{w}_h|_{H^1(\Omega_f^t)}} \leq C|\bu_{h,2}|_{H^1(\Omega_f^t)}.
    \end{equation}
    Hence,
    \begin{equation}
        |\bu_{h,2}|_{H^1(\Omega_f^t)} \leq C|\bu_{2}|_{H^1(\Omega_f^t)}.
    \end{equation}
    We also have the following estimate for the $L^2$ norm:
    \begin{equation}
        \begin{aligned}
        \|\bu_{h,2} - \bu_2\|_{L^2(\Omega_f^t)} &\leq Ch|\bu_2|_{H^1(\Omega_f^t)} \leq Ch(\|\tilde{S}q_h\|_{L^2(\Omega_f^t)} + \|\nabla\cdot\bu_{h,1}\|_{L^2(\Omega_f^t)}) \\
        &\leq C(\|\tilde{S}q_h\|_{H^{-1}(\Omega_f^t)} + \|\nabla\cdot\bu_{h,1}\|_{H^{-1}(\Omega_f^t)}).
        \end{aligned}
    \end{equation}
    Thus,
    \begin{equation}
        \|\bu_{h,2}\|_{L^2(\Omega_f^t)}\leq C(\|\bu_2\|_{L^2(\Omega_f^t)} + \|\tilde{S}q_h\|_{H^{-1}(\Omega_f^t)} + \|\nabla\cdot\bu_{h,1}\|_{H^{-1}(\Omega_f^t)}).
    \end{equation}
    Combining the above estimates with those from the continuous case, we obtain
    \begin{equation}
        \|\bu_{h,2}\|_V^2 \leq C (\|{q}_h\|_Q^2 + \|{\boldsymbol{\mu}}_h\|_{\boldsymbol{\Lambda}}^2).
    \end{equation}
    Finally, we set $\bu_h = \bu_{h,1} + \bu_{h,2}$, ${\bU}_h = \mathbf{0}$, and ${\bomega}_h = \mathbf{0}$. Then
    \begin{equation}
        \begin{aligned}
            &{b}^t({q}_h,\bu_h)+{c}^t({\boldsymbol{\mu}}_h,\bu_h - {\bU}_h - {\bomega}_h\times\mathbf{r}) \geq \|{q}_h\|_Q^2 + \|{\boldsymbol{\mu}}_h\|_{\boldsymbol{\Lambda}}^2,\\
            &\|\bu_h\|_V^2 + \|{\bU}_h\|_{T_M}^2 + \|{\bomega}_h\|_{R_M}^2 \leq C(\|{q}_h\|_Q^2 + \|{\boldsymbol{\mu}}_h\|_{\boldsymbol{\Lambda}}^2).
        \end{aligned}
    \end{equation}
    This completes the proof of the LBB condition for the system~\eqref{eq:linearised_discretized}.
\end{proof}

\section{Numerical results}
In this section, we present numerical results to evaluate the performance of the proposed numerical scheme. We begin with a convergence study for the benchmark problem of flow around a cylinder, which illustrates the necessity of using isoparametric elements in domains with curved boundaries. Next, we provide several numerical examples to demonstrate the feasibility of the scheme in simulating FSI problems involving a moving rigid body. We then apply the method to a representative problem arising in a DLD device, showcasing its capability in handling complex FSI scenarios. Finally, we examine the efficiency of the preconditioner introduced in the previous section.

\subsection{Convergence study}

\subsubsection{Flow around a cylinder}

We first investigate the spatial convergence of the numerical scheme on a problem with a curved boundary, using the classical benchmark of flow around a cylinder \cite{schafer1996benchmark}. The computational domain and boundary conditions are shown in Figure~\ref{fig:Cylinder}.

\begin{figure}[H]
\centering
\includegraphics[width=0.8\textwidth]{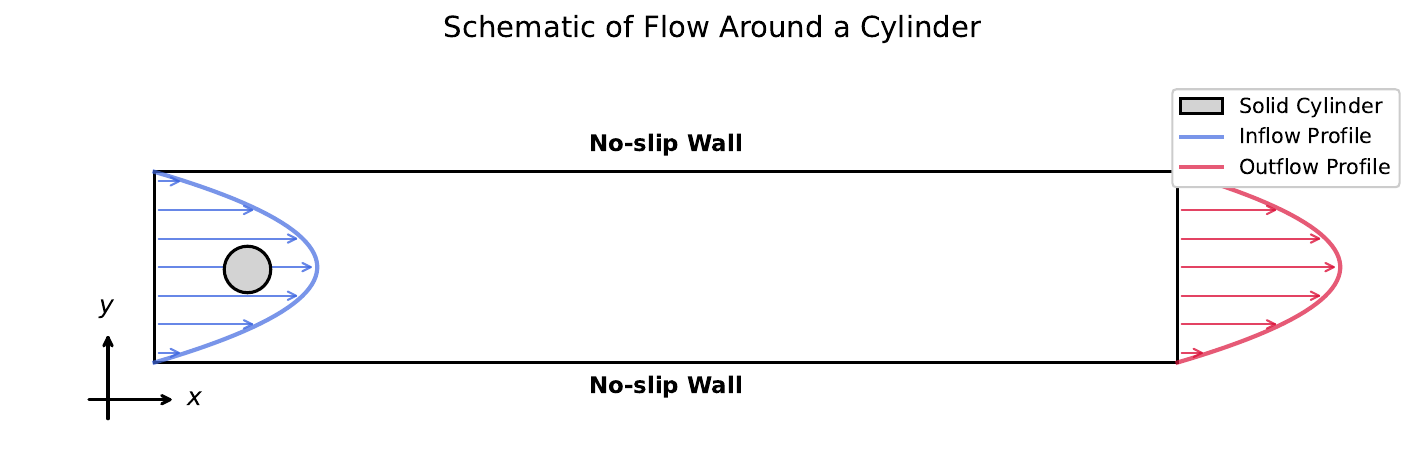}
\caption{Geometry and boundary conditions for flow around a cylinder.}
\label{fig:Cylinder}
\end{figure}

The domain is $\Omega = [0,2.2] \times [0,0.41]$, containing a cylinder of radius $r = 0.05$ centered at $(0.2, 0.2)$. The fluid density is $\rho_f = 1$ and the dynamic viscosity is $\mu_f = 0.001$. A parabolic inflow and outflow profile $\bu_f = \left(4u_m y(0.41 - y)/0.41^2,\, 0\right)$ is imposed on the left and right boundaries, with $u_m = 0.3$. No-slip boundary conditions are applied on the top and bottom walls and on the cylinder surface. The initial velocity field is set to zero. We compute the steady-state solution at a Reynolds number $\Reynolds = 20$, defined as $\Reynolds = \rho_f \bar{u} d / \mu_f$, where $d = 2r$ is the cylinder diameter and $\bar{u}$ is the average inflow velocity.

\begin{remark}
    We avoid the use of the do-nothing boundary condition at the outflow, as it may induce velocity singularities at the corners where the outflow meets the no-slip walls. For a detailed discussion, see \cite{he2024analytic,guo2006analytic} and references therein.
\end{remark}

To assess spatial convergence, we perform a mesh refinement study by successively halving the mesh size $h$. The velocity errors are measured in the $H^1(\Omega)$ and $L^2(\Omega)$ norms. We compare the performance of the standard Taylor-Hood element ($\mathbb{P}_2/\mathbb{P}_1$) with that of its isoparametric counterpart. The results are reported in Table~\ref{tab:convergence_cylinder} and depicted in Figure~\ref{fig:convergence_cylinder}, where the reference solution $\bu_{ref}$ is obtained on a highly refined mesh with $h = 0.0015625$.

\begin{figure}[H]
\centering
\includegraphics[width=0.8\textwidth]{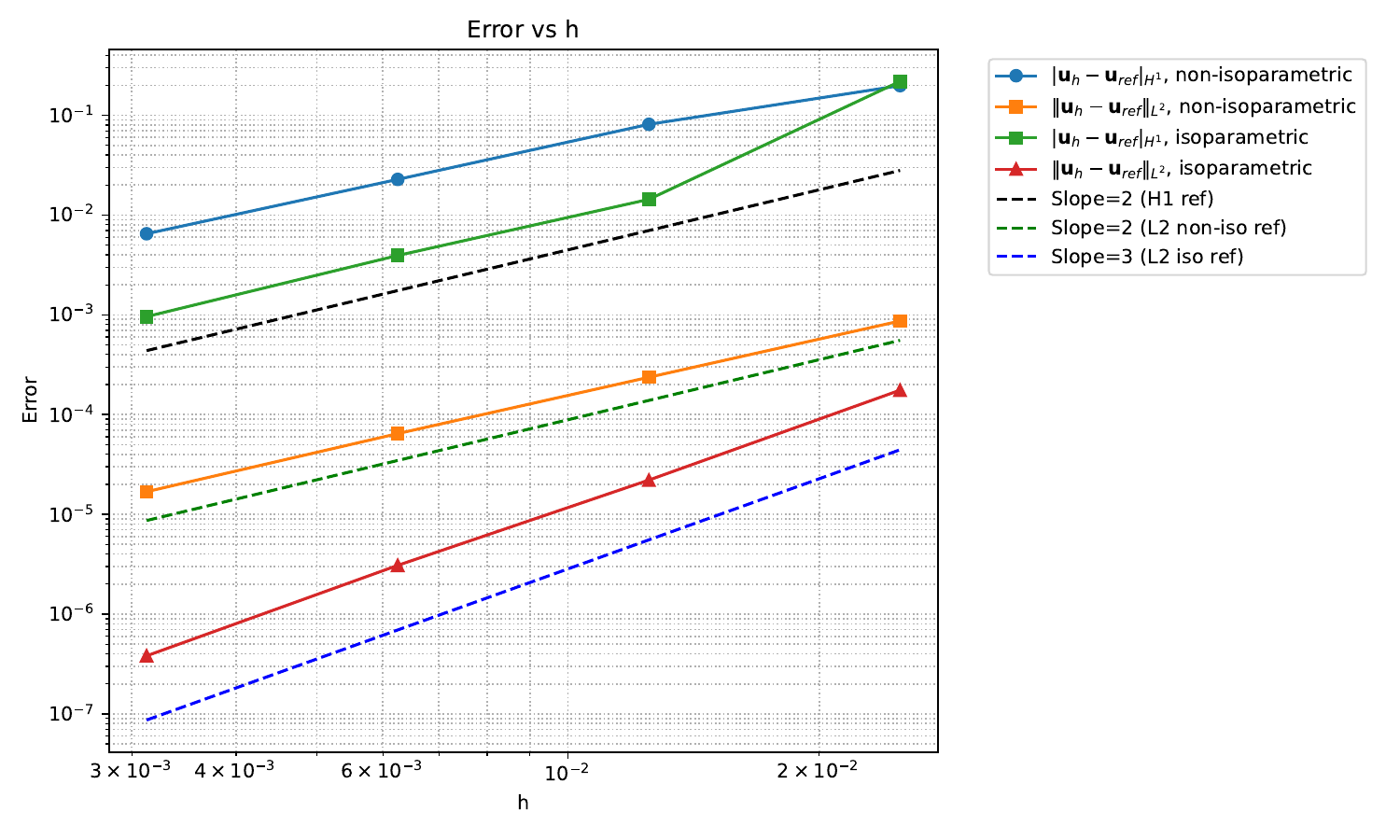}
\caption{Convergence of the velocity for flow around a cylinder.}
\label{fig:convergence_cylinder}
\end{figure}

\begin{table}[htbp]
\centering
\caption{Convergence results for flow around a cylinder using standard Taylor-Hood and isoparametric elements.}
\label{tab:convergence_cylinder}
\begin{tabular}{c 
                S[table-format=0.4] c 
                S[table-format=1.4e-1] c 
                S[table-format=0.4] c 
                S[table-format=1.4e-1] c}
\toprule
& \multicolumn{4}{c}{Standard Taylor-Hood element} & \multicolumn{4}{c}{Isoparametric element} \\
\cmidrule(lr){2-5} \cmidrule(lr){6-9}
{$h$} & {$|\bu_h - \bu_{ref}|_{H^1}$} & {order} & {$\|\bu_h - \bu_{ref}\|_{L^2}$} & {order} & {$|\bu_h - \bu_{ref}|_{H^1}$} & {order} & {$\|\bu_h - \bu_{ref}\|_{L^2}$} & {order} \\
\midrule
0.025    & 0.1985    & {---} & 8.6753e-04 & {---} & 0.2189    & {---} & 1.7570e-04 & {---} \\
0.0125   & 0.0813    & 1.29  & 2.3682e-04 & 1.87  & 0.0144    & 3.93  & 2.2074e-05 & 2.99  \\
0.00625  & 0.0228    & 1.83  & 6.4496e-05 & 1.88  & 0.0039    & 1.87  & 3.0817e-06 & 2.84  \\
0.003125 & 0.0065    & 1.81  & 1.6876e-05 & 1.93  & 0.0009    & 2.04  & 3.8174e-07 & 3.01  \\
\bottomrule
\end{tabular}
\end{table}

The results clearly show that the standard Taylor-Hood element fails to achieve optimal convergence rates due to the geometric approximation error on the curved boundary. In contrast, the isoparametric formulation preserves optimal convergence in velocity, confirming its necessity for problems involving curved interfaces.

\subsubsection{Particle in shear flow}

We consider a benchmark problem involving a particle suspended in a shear flow, previously studied in \cite{krause2017incompressible,hollbacher2020gradient}. The purpose of this test is to verify the ability of the numerical scheme to accurately capture the rotational motion of a rigid particle under shear.

The computational domain is $\Omega = [-3, 3] \times [-2, 2]$, with a circular particle of radius $r = 0.9$ initially centered at the origin $(0, 0)$. Both the fluid and solid densities are set to $\rho_f = \rho_s = 1$, and the fluid viscosity is $\mu_f = 0.01$. A shear velocity profile $\bu_f = (u_s y, 0)$ is imposed on the top and bottom boundaries where $u_s$ is a constant, while no-slip conditions are applied on the particle surface. The initial translational and angular velocities of the particle are zero.

To assess temporal and spatial convergence, we perform a refinement study by successively halving both the mesh size $h$ and the time step $\Delta t$. We set $u_s = 0.02$, following \cite{hollbacher2020gradient}, for which the analytical angular velocity of the particle is known to be $\bomega = 0.005$.

We define the root-mean-square (RMS) error in the angular velocity as:
\begin{equation}
    e_{\bomega,h} = \sqrt{\frac{1}{N_t} \sum_{i=1}^{N_t} \left| \bomega_h(t_i) - \bomega_{\text{ref}}(t_i) \right|^2},
\end{equation}
where $\bomega_h(t_i)$ denotes the computed angular velocity at time $t_i$, and $\bomega_{\text{ref}}$ is the reference solution obtained using the finest mesh and smallest time step (i.e., $h = 0.075$, $\Delta t = 0.3125$).

The convergence results are presented in Table~\ref{tab:convergence_shear} and Figure~\ref{fig:convergence_shear}. We observe that the first-order scheme achieves approximately first-order convergence in time, while the second-order scheme exhibits more than second-order convergence, likely due to the dominance of spatial errors at these parameter settings. This confirms the correctness and accuracy of the temporal discretization implementation.

\begin{figure}[htbp]
\centering
\includegraphics[width=0.7\textwidth]{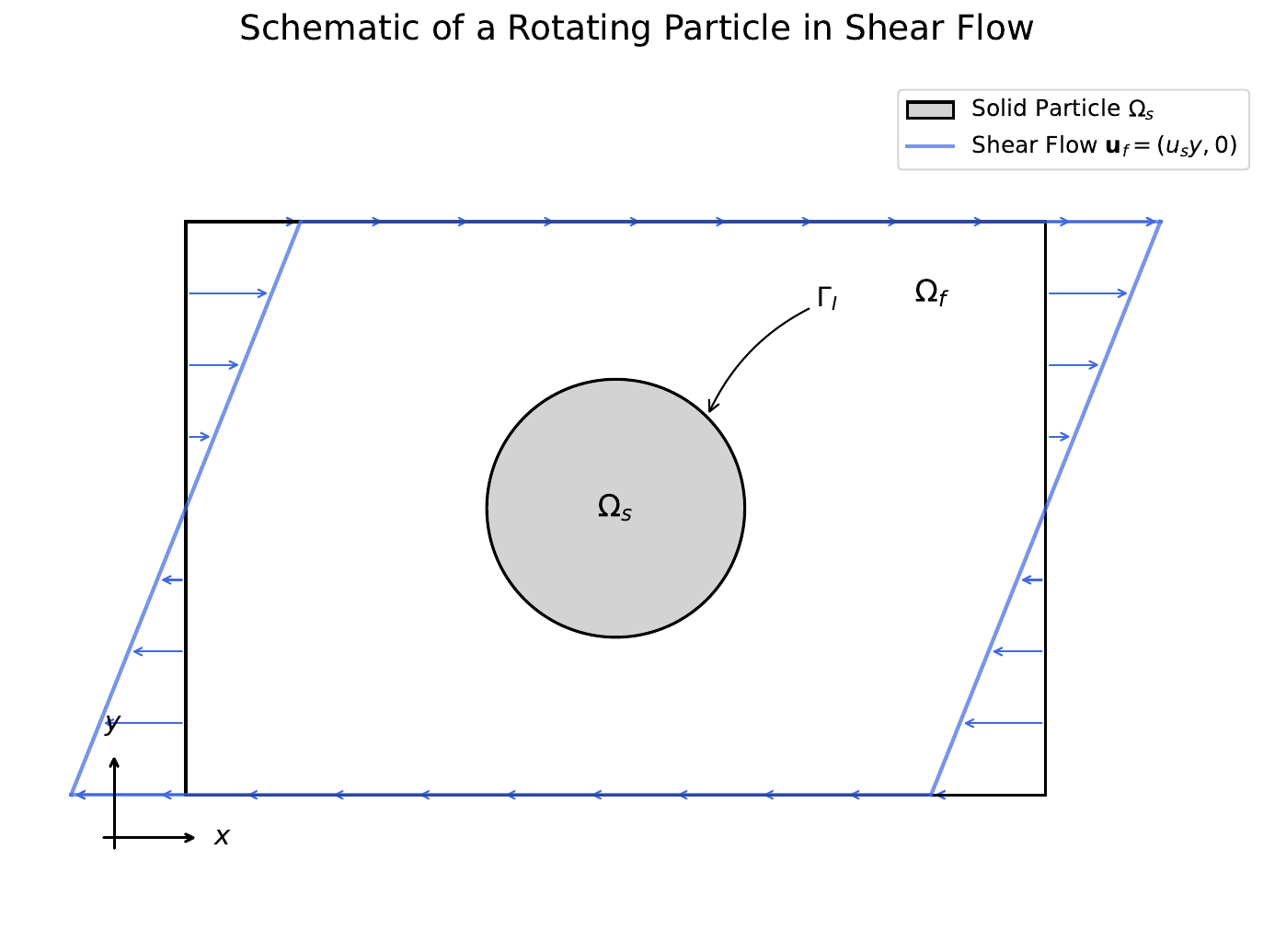}
\caption{Geometry and boundary conditions for the shear flow problem.}
\label{fig:Shear}
\end{figure}

\begin{table}[htbp]
\centering
\caption{Convergence of the angular velocity of the particle in shear flow.}
\label{tab:convergence_shear}
\begin{tabular}{
    S[table-format=1.3]
    S[table-format=1.2e-1] c
    S[table-format=1.2e-1] c
}
\toprule
& \multicolumn{2}{c}{PRK1} & \multicolumn{2}{c}{PRK2} \\
\cmidrule(lr){2-3} \cmidrule(lr){4-5}
{$\Delta t$} & {$e_{\boldsymbol{\omega},h}$} & {order} & {$e_{\boldsymbol{\omega},h}$} & {order} \\
\midrule
2.5    & 4.26e-05 & {---} & 1.19e-06 & {---} \\
1.25   & 2.22e-05 & 0.94  & 2.64e-07 & 2.17  \\
0.625  & 1.14e-05 & 0.97  & 5.19e-08 & 2.35  \\
\bottomrule
\end{tabular}
\end{table}

\begin{figure}[htbp]
\centering
\includegraphics[width=0.7\textwidth]{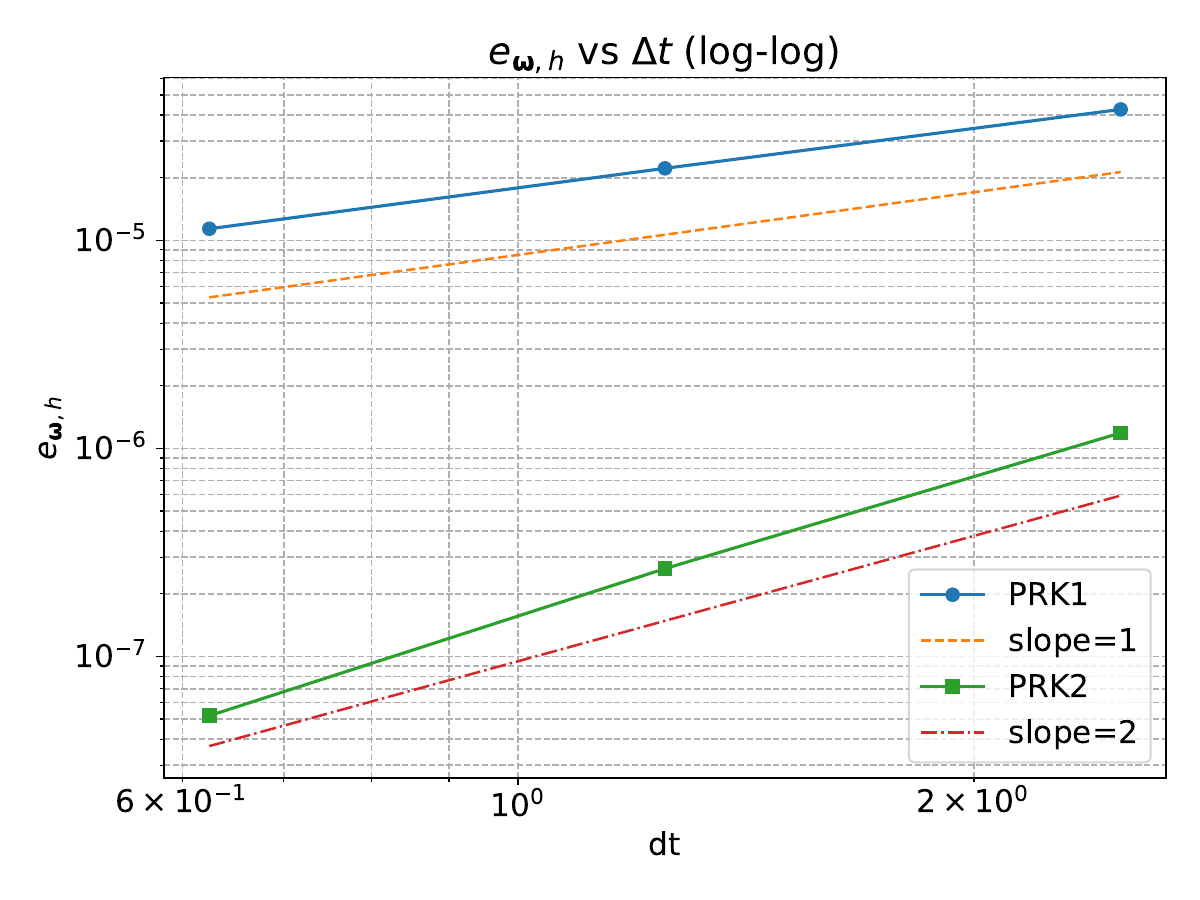}
\caption{Convergence of the angular velocity error in shear flow.}
\label{fig:convergence_shear}
\end{figure}

\subsubsection{A sedimenting particle}
We consider a benchmark problem of a particle sedimenting under gravity in a vertical channel, designed to validate the capability of the numerical scheme in capturing translational motion in fluid-structure interaction (FSI) problems. The geometry and boundary conditions are illustrated in Figure~\ref{fig:Sedimentation}.

\begin{figure}[htbp]
\centering
\includegraphics[width=0.6\textwidth]{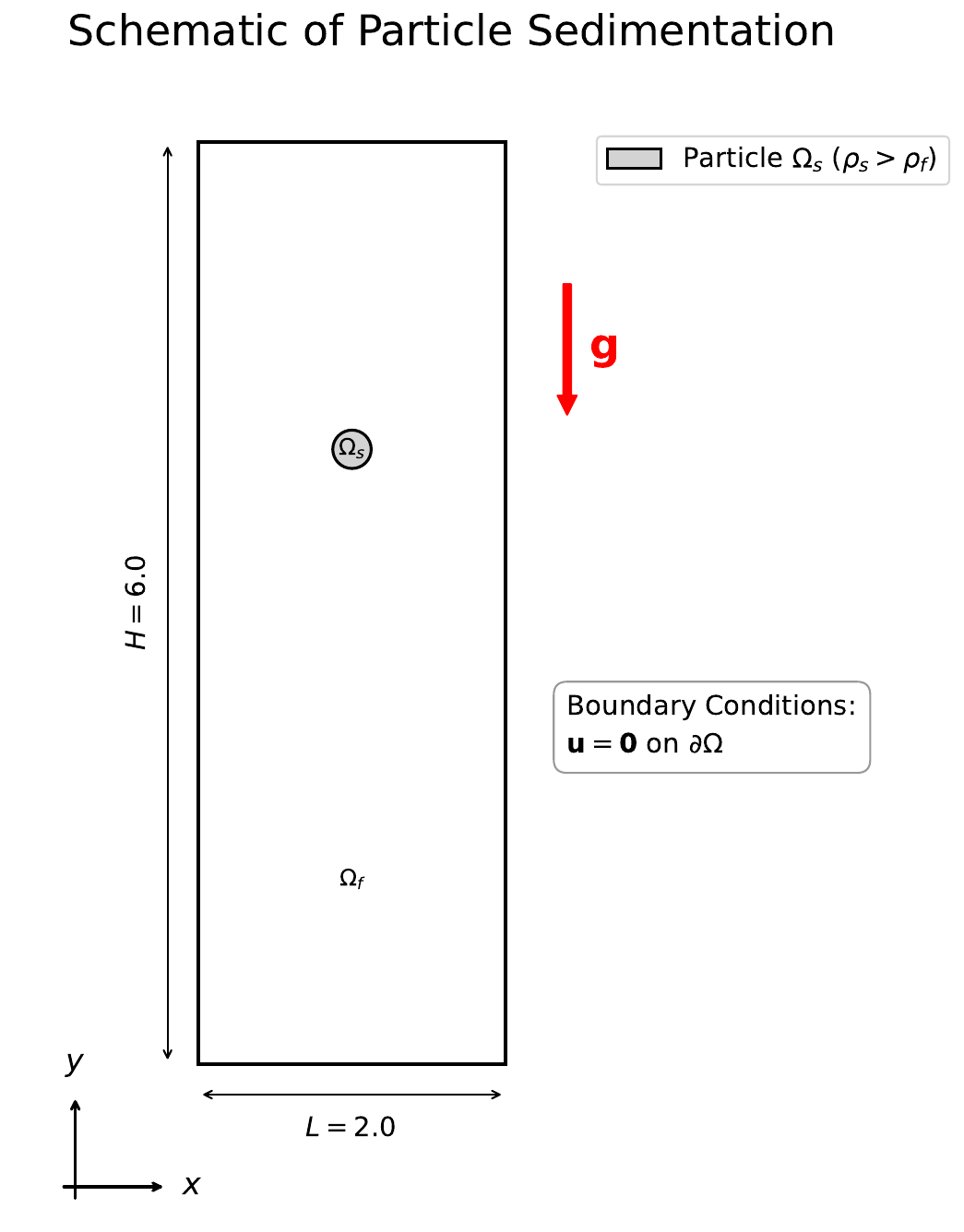}
\caption{Geometry and boundary conditions for the particle sedimentation problem.}
\label{fig:Sedimentation}
\end{figure}

The computational domain is $\Omega = [0,2] \times [0,6]$. At $t=0$, a circular particle of diameter $d = 0.25$ and density $\rho_s = 1.25$ is placed at $(1,4)$, settling through an incompressible fluid with density $\rho_f = 1$ and viscosity $\mu_f = 0.1$. Gravity drives the motion in the negative $y$-direction. We focus on the vertical displacement $\mathbf{x}_y$ and vertical velocity $\mathbf{U}_y$ of the particle.

Due to the finite width $D = 2$ of the channel, wall effects influence the terminal velocity. Following \cite{hollbacher2019rotational}, the asymptotic wall correction force $F_w(d, D)$ and the stationary sedimentation velocity $\mathbf{U}_{st}$ are given by:
\begin{equation}
    \begin{aligned}
        F_w(d, D) &= \ln \left(\frac{D}{d}\right) - 0.9157 + 1.7244\left(\frac{d}{D}\right)^2 - 1.7302\left(\frac{d}{D}\right)^4 + O\left(\left(\frac{d}{D}\right)^6\right), \\
        \mathbf{U}_{st} &= \frac{(\rho_s - \rho_f) d^2}{16 \mu_f} F_w(d, D) \, g,
    \end{aligned}
\end{equation}
where $g = 9.81$ is the gravitational acceleration. The theoretical terminal velocity is $\mathbf{U}_{st} \approx 0.114$ in magnitude.
We observe that the particle velocity approaches steady state by $t = 5$. Figure~\ref{fig:Sedimentation-5s} shows the vertical position and velocity at this time, computed using different time steps. The results indicate good agreement with the theoretical value.

\begin{figure}[htbp]
    \begin{subfigure}{0.49\textwidth}
        \centering
        \includegraphics[width=0.95\textwidth]{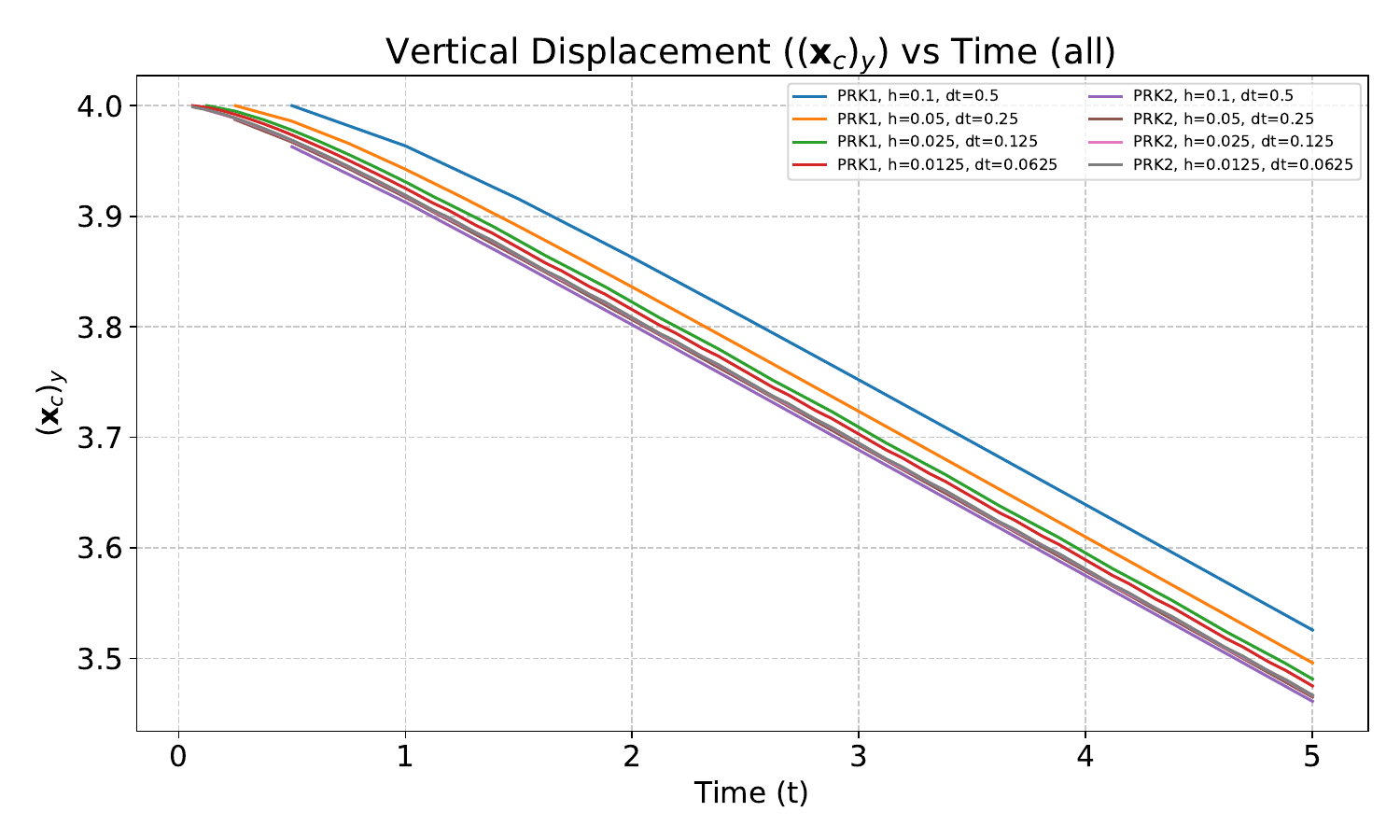}
        \caption{Evolution of the vertical position of the particle.}
    \end{subfigure}
    \begin{subfigure}{0.49\textwidth}
        \centering
        \includegraphics[width=0.95\textwidth]{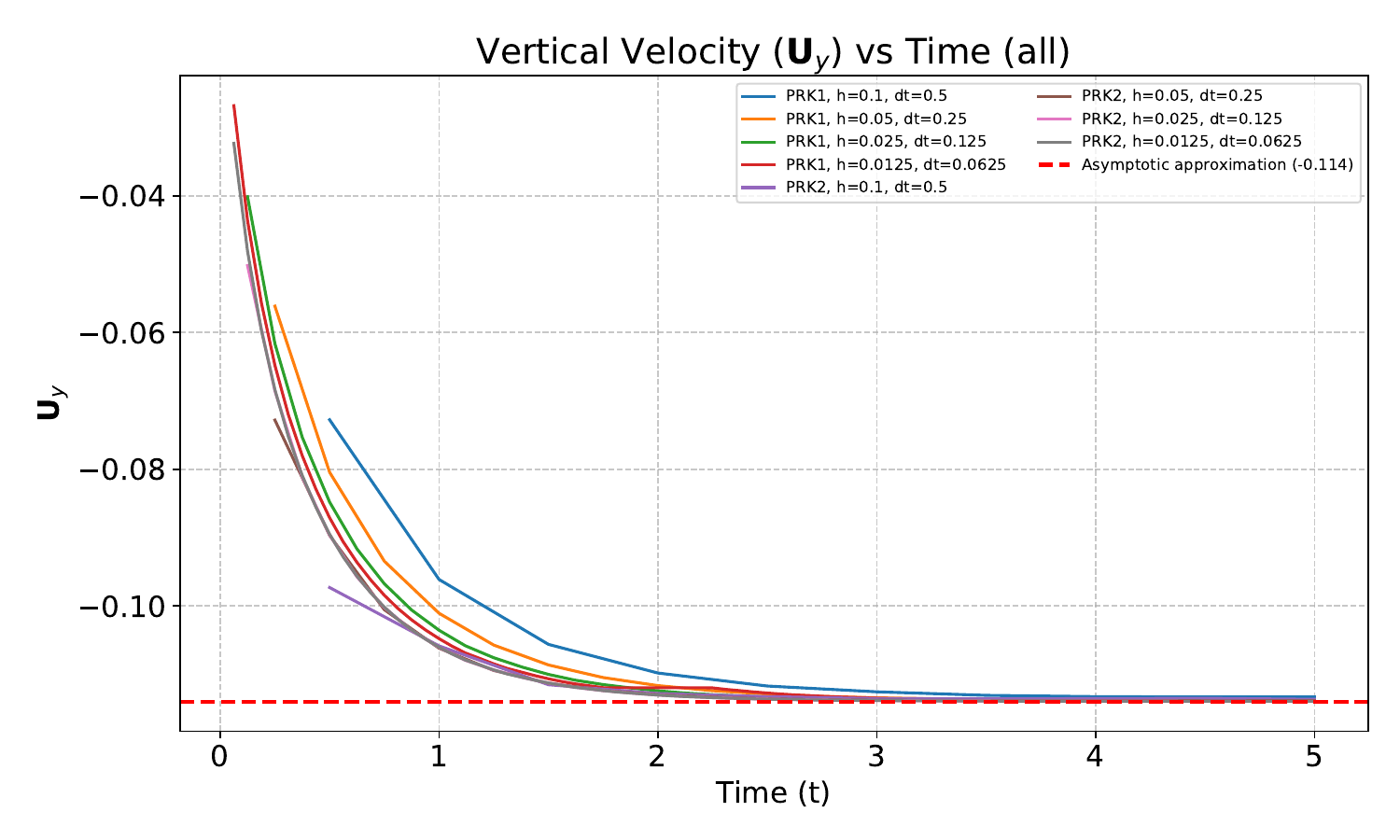}
        \caption{Evolution of the vertical velocity of the particle.}
    \end{subfigure}
    \caption{Evolution of the vertical position and velocity of the sedimenting particle for different time steps. The dashed line indicates the theoretical terminal velocity.}
    \label{fig:Sedimentation-5s}
\end{figure}

To assess convergence at fixed time, we compute errors in both position and velocity relative to a reference solution obtained on the finest mesh ($h = 0.05$) and smallest time step ($\Delta t = 0.0625$). The convergence results at $t = 5$ are reported in Table~\ref{tab:sedimentation_convergence_t=5}. We observe first-order convergence for the PRK1 scheme  and nearly second-order convergence for the PRK2 scheme in the position. The velocity exhibits higher-than-expected convergence rates, which we attribute to the near-steady state where temporal errors are small and spatial discretization dominates.

\begin{table}[htbp]
\centering
\caption{Convergence of position and velocity at $t = 5$ seconds.}
\label{tab:sedimentation_convergence_t=5}
\begin{tabular}{
    S[table-format=0.3]
    S[table-format=1.4]    c
    S[table-format=1.4]    c
    S[table-format=1.2e-1] c
    S[table-format=-1.4]   c
    S[table-format=1.2e-1] c
    S[table-format=1.2e-1] c
}
\toprule
\multicolumn{7}{c}{PRK1} \\
\midrule
{$\Delta t$} & {$\mathbf{x}_y$} & {error} & {order} & {$\mathbf{U}_y$} & {error} & {order} \\
\midrule
0.5    & 3.5258 & 0.0589 & {---} & -0.1133 & 6.41e-04 & {---} \\
0.25   & 3.4960 & 0.0292 & 1.01  &  0.1138 & 1.48e-04 & 2.11  \\
0.125  & 3.4814 & 0.0146 & 1.00  & -0.1139 & 4.67e-05 & 1.67  \\
\addlinespace
\toprule
\multicolumn{7}{c}{PRK2} \\
\midrule
{$\Delta t$} & {$\mathbf{x}_y$} & {error} & {order} & {$\mathbf{U}_y$} & {error} & {order} \\
\midrule
0.5    & 3.4613 & 0.0056 & {---} & -0.1140 & 2.96e-04 & {---} \\
0.25   & 3.5181 & 0.0017 & 1.71  & -0.1140 & 8.70e-05 & 2.88  \\
0.125  & 3.4664 & 0.0004 & 2.01  & -0.1140 & 1.20e-05 & 1.77  \\
\bottomrule
\end{tabular}
\end{table}

To evaluate the temporal convergence order more rigorously, we compute the error over the entire time interval $[0, 5]$, using the finest-resolution solution as the reference. The RMS errors for velocity and angular velocity are defined as:
\begin{equation}
    \begin{aligned}
        e_{\mathbf{U},h} &= \sqrt{\frac{1}{N_t} \sum_{i=1}^{N_t} \left| \mathbf{U}_h(t_i) - \mathbf{U}_{\text{ref}}(t_i) \right|^2}, \\
        e_{\boldsymbol{\omega},h} &= \sqrt{\frac{1}{N_t} \sum_{i=1}^{N_t} \left| \boldsymbol{\omega}_h(t_i) - \boldsymbol{\omega}_{\text{ref}}(t_i) \right|^2}.
    \end{aligned}
\end{equation}

The results are presented in Table~\ref{tab:convergence_sedimentation}. The velocity error converges at the expected rate (first-order for the PRK1 scheme, second-order for the PRK2 scheme). The angular velocity error converges at a higher rate, which may be attributed to the symmetry of the flow and the high accuracy of the boundary approximation, despite the absence of physical rotation — the small computed $\boldsymbol{\omega}$ arises purely from numerical asymmetry.

\begin{table}[htbp]
\centering
\caption{Temporal convergence of velocity and angular velocity over $[0, 5]$.}
\label{tab:convergence_sedimentation}
\begin{tabular}{c 
                S[table-format=1.2e-1] c 
                S[table-format=1.2e-1] c 
                S[table-format=1.2e-1] c 
                S[table-format=1.2e-1] c}
\toprule
& \multicolumn{4}{c}{PRK1} & \multicolumn{4}{c}{PRK2} \\
\cmidrule(lr){2-5} \cmidrule(lr){6-9}
{$\Delta t$} 
& {$e_{\mathbf{U},h}$} & {order} & {$e_{\boldsymbol{\omega},h}$} & {order} 
& {$e_{\mathbf{U},h}$} & {order} & {$e_{\boldsymbol{\omega},h}$} & {order} \\
\midrule
0.5    & 5.34e-02 & {---} & 6.38e-05 & {---} & 6.12e-03 & {---} & 6.14e-05 & {---} \\
0.25   & 2.67e-02 & 1.01  & 7.39e-06 & 3.11  & 1.82e-03 & 1.75  & 2.98e-06 & 4.36  \\
0.125  & 1.34e-02 & 1.00  & 2.78e-07 & 4.73  & 4.33e-04 & 2.07  & 3.61e-07 & 3.04  \\
\bottomrule
\end{tabular}
\end{table}

\begin{figure}[htbp]
    \centering
    \includegraphics[width=0.7\textwidth]{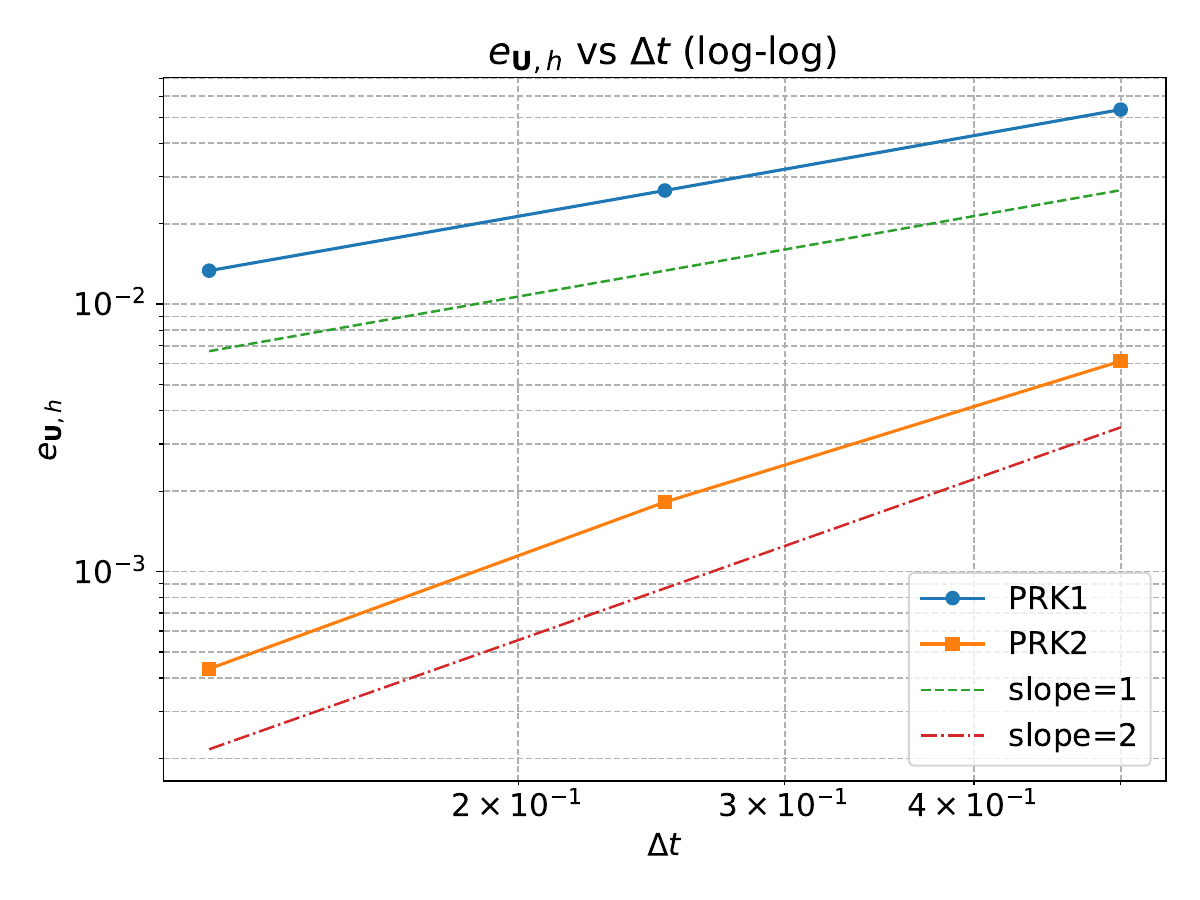}
    \caption{Convergence of the velocity error $e_{\mathbf{U},h}$ in the particle sedimentation problem.}
    \label{fig:convergence_V}
\end{figure}

\begin{figure}[htbp]
    \centering
    \includegraphics[width=0.7\textwidth]{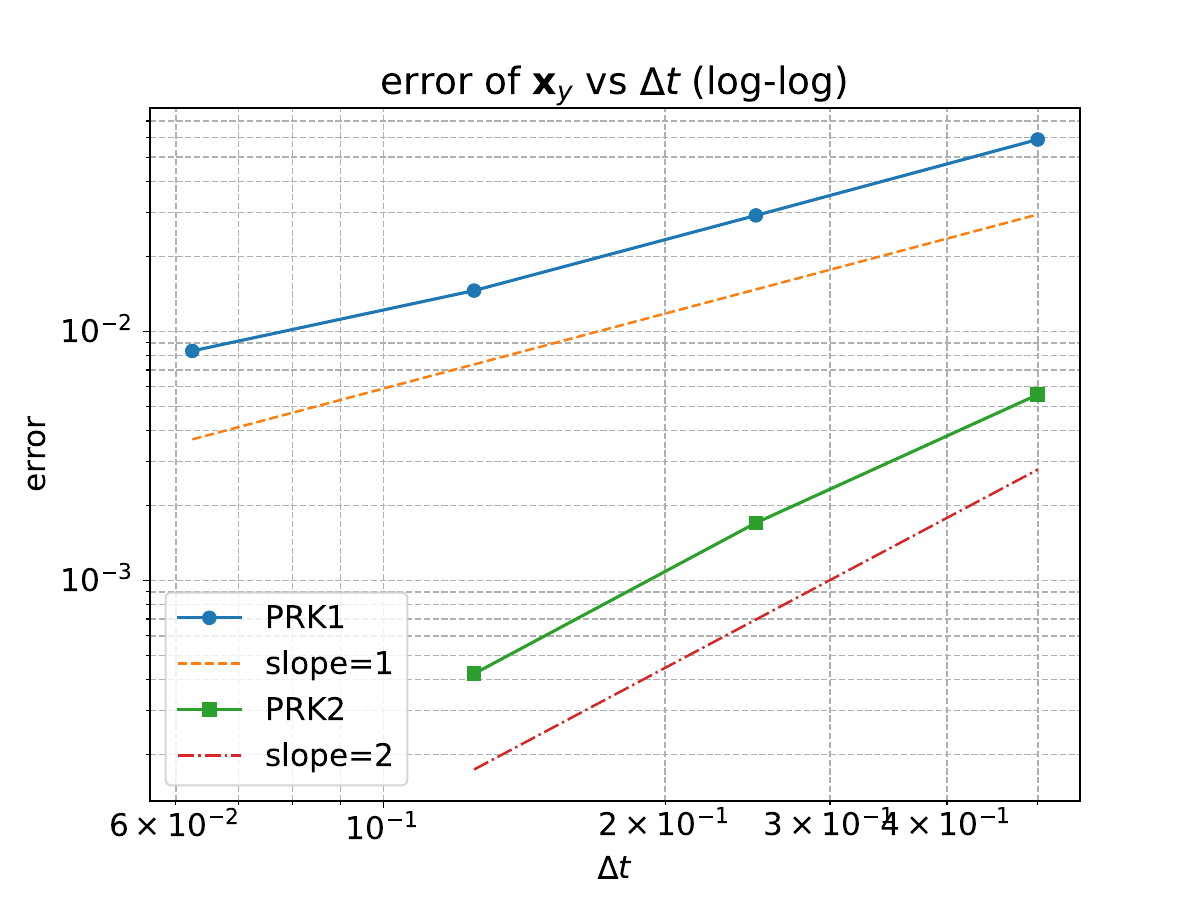}
    \caption{Convergence of the vertical position error at $t = 5$ in the particle sedimentation problem.}
    \label{fig:convergence_x}
\end{figure}

\subsubsection{A test case with two pillars}\label{DLD-B1}

We consider a benchmark configuration involving two pillars and a circular particle in a microfluidic setting, representative of DLD devices. The geometry and boundary conditions are illustrated in Figure~\ref{fig:DLD-B1}. The fluid density is $\rho_f = 1.0\times10^{-12}\ \mathrm{g/\mu m^3}$ and the dynamic viscosity is $\mu_f = 1.0\times10^{-9}\ \mathrm{g/(\mu m \cdot ms)}$. With the characteristic length taken as the pillar diameter $L = 35\ \mathrm{\mu m}$ and the characteristic velocity $u_0 = 30\ \mathrm{\mu m/ms}$, the resulting Reynolds number is $\Reynolds = \rho_f u_0 L / \mu_f = 1.05$.

\begin{figure}[htbp]
\centering
\includegraphics[width=0.8\textwidth]{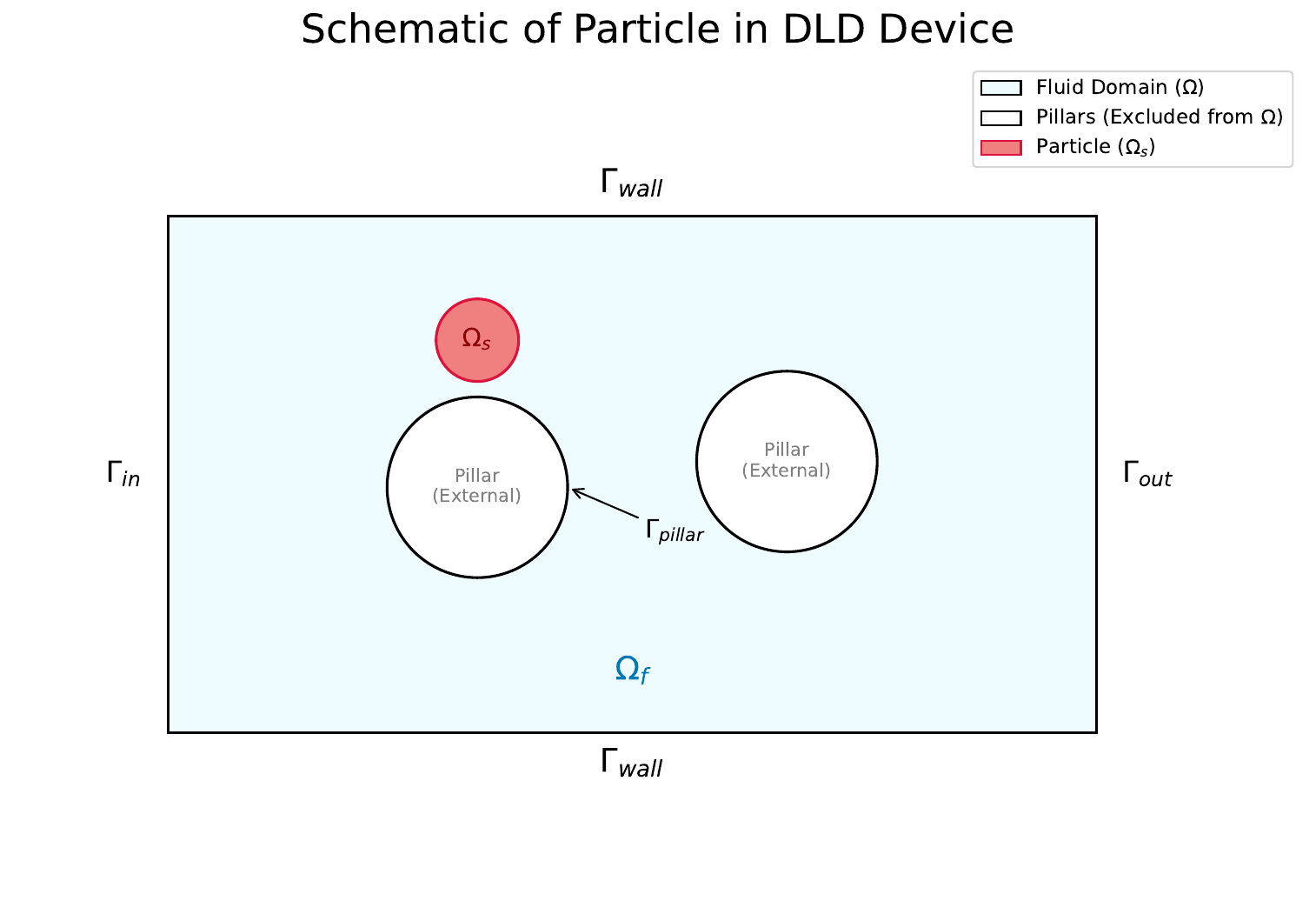}
\caption{Geometry and boundary conditions for the DLD-B1 test case.}
\label{fig:DLD-B1}
\end{figure}

A time-dependent inflow velocity is prescribed on the inlet $\Gamma_{\mathrm{in}}$:
\begin{equation}\label{eq:inflow-B1}
    \mathbf{u} = \left\{
    \begin{aligned}
    &\frac12\left[\sin\left(\left(\frac{t}{T} - 0.5\right)\pi\right) + 1\right] u_0 \frac{y(W - y)}{W^2} \mathbf{e}_x, && \text{if } t \leq T, \\
    & u_0 \frac{y(W - y)}{W^2} \mathbf{e}_x, && \text{if } t > T,
    \end{aligned}
    \right.
\end{equation}
where $T = 0.1\ \mathrm{ms}$. No-slip conditions are applied on the channel walls $\Gamma_{\mathrm{wall}}$ and the pillar boundaries $\Gamma_{\mathrm{pillar}}$. The particle, with radius $R = 8.0\ \mathrm{\mu m}$ and density matching the fluid (neutrally buoyant), is initially placed at $(60, 76)\ \mathrm{\mu m}$. Both the fluid and particle velocities are initialized to zero. The simulation terminates when the $x$-coordinate of the particle's center of mass exceeds $160\ \mathrm{\mu m}$.

To assess temporal convergence, we perform a refinement study with fixed mesh size $h = 1\,\mu\mathrm{m}$ and varying time steps $\Delta t$. The reference solution is taken as the numerical result computed with the smallest time step ($\Delta t = 0.00625\,\mathrm{ms}$). We use the root-mean-square (RMS) errors over the full simulation interval $[0, T_{\mathrm{end}}]$ with the same definitions as before.

The convergence results are presented in Table~\ref{tab:convergence} and visualized in Figure~\ref{fig:convergence}. We observe that the particle trajectory $\mathbf{x}_c$, which is the primary quantity of interest in DLD applications, converges at the expected rate: approximately first-order for the PRK1 scheme and second-order for the PRK2 scheme. In contrast, the convergence rates for velocity $\mathbf{U}$ and angular velocity $\boldsymbol{\omega}$ are lower than optimal, particularly for the second-order method.  Such behavior is commonly observed in FSI simulations, where derived quantities often exhibit reduced convergence. 


\begin{table}[htbp]
\centering
\caption{Temporal convergence analysis for the DLD-B1 test case. Errors are measured in RMS norm over the simulation interval.}
\label{tab:convergence}
\begin{tabular}{c
                S[table-format=1.4e-1] c
                S[table-format=1.3e-1] c
                S[table-format=1.3e-1] c}
\toprule
\multicolumn{7}{c}{PRK1} \\
\midrule
{$\Delta t$ (ms)} & {$e_{\mathbf{x}_c,h}$ ($\mu$m)} & {order} & {$e_{\mathbf{U},h}$ ($\mu$m/ms)} & {order} & {$e_{\boldsymbol{\omega},h}$ (rad/ms)} & {order} \\
\midrule
0.1     & 9.574e-02 & {---} & 1.36e-03 & {---} & 1.34e-02 & {---} \\
0.05    & 4.230e-02 & 1.18  & 7.49e-04 & 0.86  & 1.02e-02 & 0.40  \\
0.025   & 1.951e-02 & 1.12  & 4.67e-04 & 0.68  & 8.60e-03 & 0.24  \\
0.0125  & 9.261e-03 & 1.07  & 3.34e-04 & 0.48  & 7.84e-03 & 0.13  \\
\addlinespace
\toprule
\multicolumn{7}{c}{PRK2} \\
\midrule
{$\Delta t$ (ms)} & {$e_{\mathbf{x}_c,h}$ ($\mu$m)} & {order} & {$e_{\mathbf{U},h}$ ($\mu$m/ms)} & {order} & {$e_{\boldsymbol{\omega},h}$ (rad/ms)} & {order} \\
\midrule
0.1     & 2.170e-02 & {---} & 6.27e-04 & {---} & 3.65e-03 & {---} \\
0.05    & 4.378e-03 & 2.31  & 3.23e-04 & 0.96  & 1.74e-03 & 1.07  \\
0.025   & 7.084e-04 & 2.63  & 1.75e-04 & 0.88  & 8.11e-04 & 1.10  \\
0.0125  & 1.355e-04 & 2.39  & 8.56e-05 & 1.03  & 3.48e-04 & 1.22  \\
\bottomrule
\end{tabular}
\end{table}

\begin{figure}[htbp]
\centering
\includegraphics[width=0.7\textwidth]{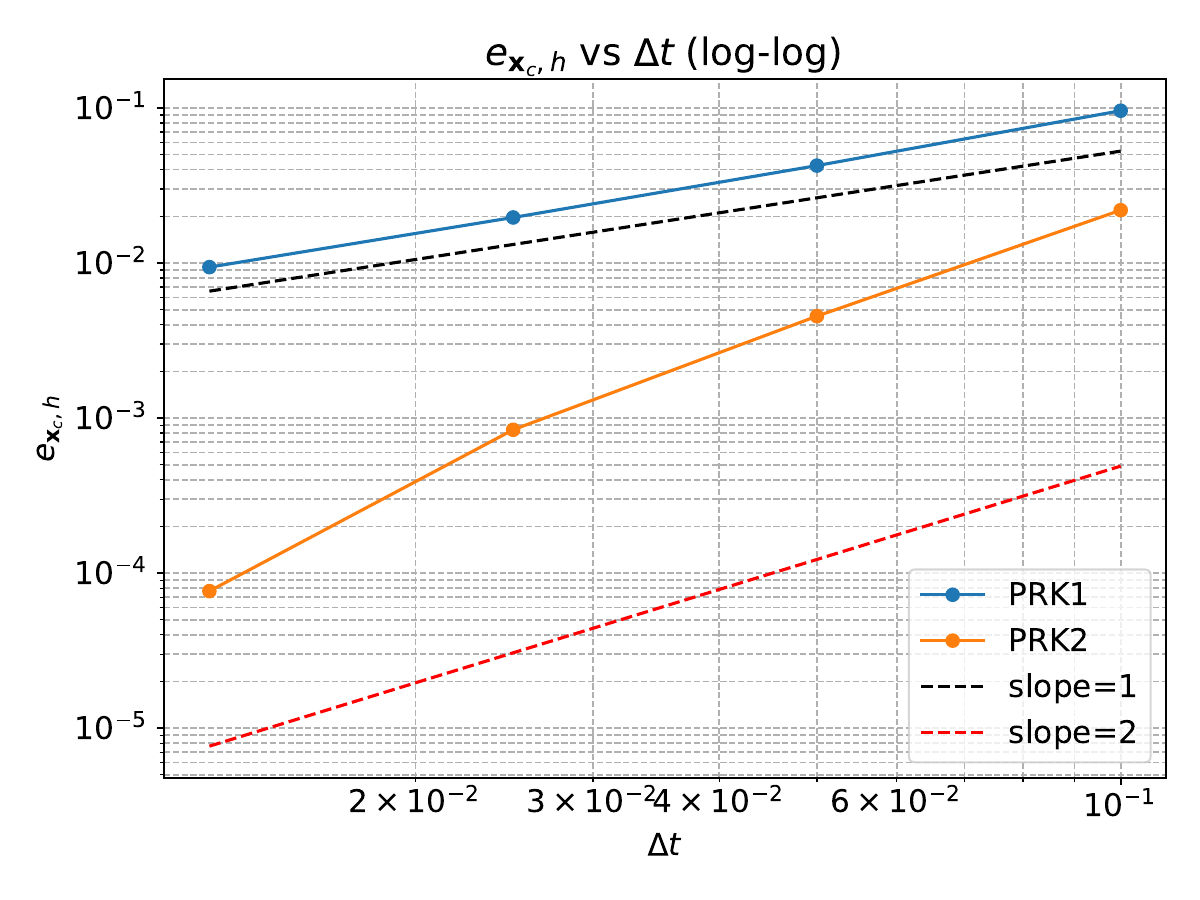}
\caption{Convergence of the RMS error in particle trajectory $e_{\mathbf{x}_c}$ in the DLD test case.}
\label{fig:convergence}
\end{figure}

\subsection{Performance of the preconditioner}

We evaluate the efficiency of the proposed preconditioner $\mathbf{P}$ defined in \eqref{eq:preconditioner} for solving the linearized discrete system~\eqref{eq:linearised_discretized}.

We apply the GMRES method to solve this system and assess the convergence behavior under mesh and time-step refinement. The iteration is terminated when the relative residual satisfies
\begin{equation}
    \frac{\|\mathbf{r}_k\|}{\|\mathbf{b}\|} < 10^{-6},
\end{equation}
where $\mathbf{r}_k$ is the residual at the $k$-th iteration and $\mathbf{b}$ is the right-hand side vector.

Since the nonlinear system is solved using Newton's method, we define $N_{\text{GMRES}}$ as the average number of GMRES iterations per Newton iteration, further averaged over the first 10 time steps. This metric reflects the overall solver efficiency during the initial transient phase, where the solution evolves rapidly and coupling effects are most pronounced.

The test case from Section~\ref{DLD-B1} is used, with mesh sizes $h = 1, 0.5, 0.25, 0.125\,\mu\mathrm{m}$ and time steps $\Delta t = 1/640, 1/320, 1/160\,\mathrm{ms}$. The results are reported in Table~\ref{tab:gmres_refinement}, 
which confirm that $N_{\text{GMRES}}$ remains nearly constant across different mesh and time-step sizes, with values ranging between approximately 18 and 27. 


\begin{table}[htbp]
\centering
\caption{Average number of GMRES iterations ($N_{\text{GMRES}}$) for the preconditioned system under spatial and temporal refinement.}
\label{tab:gmres_refinement}
\begin{tabular}{c S[table-format=2.2] S[table-format=2.2] S[table-format=2.2]}
\toprule
{$h\,(\mu\mathrm{m})$}  & {$\Delta t = 1/160$}  & {$\Delta t = 1/320$}& {$\Delta t = 1/640$}\\
\midrule
1      & 18.40  & 17.95  & 25.00\\
0.5    & 18.74  & 21.10  & 24.00\\
0.25   & 19.05  & 21.30  & 27.70\\
0.125  & 18.50  & 22.75  & 26.90\\
\bottomrule
\end{tabular}
\end{table}

Crucially, the absence of a systematic increase in iteration count as $h \to 0$ or $\Delta t \to 0$ confirms that the preconditioner is robust and scalable. This verifies the theoretical analysis presented in Section~\ref{sec:preconditioner}, demonstrating that $\mathbf{P}$ effectively clusters the eigenvalues of the system matrix independent of discretization parameters.

\section{Conclusions}
In this work we proposed a high-order fitted-mesh DLM-ALE framework for fluid-rigid-body interaction, motivated by trajectory-sensitive transport in DLD microfluidic devices. The fitted-mesh DLM discretization preserves a sharp interface and avoids the interpolation-induced smearing of classical two-grid DLM/FD approaches.

We coupled the method with a partitioned IMEX Runge-Kutta time integrator that advances the mesh motion explicitly and treats the DLM-constrained flow subsystem implicitly, yielding high-order accurate rigid-body trajectories at reduced cost compared with fully implicit monolithic schemes. A continuous and discrete well-posedness analysis of the linearized generalized Stokes formulation provides the basis for an operator-based block preconditioner, and the numerical experiments confirm high-order convergence together with robust GMRES performance over a range of meshes and time steps.

Future work will extend the framework to deformable particles and microscale interfacial models (e.g., Navier slip), and apply it to device-scale DLD geometries and design studies.

\bibliography{reference}
\bibliographystyle{plain}

\end{document}